\documentclass[11pt, a4paper, reqno]{amsart}
\usepackage{color,latexsym,amsfonts,amssymb,amsmath, amsthm, bbm}
\usepackage{hyperref}
\usepackage{fullpage}
\usepackage{fancyhdr}
\usepackage{graphicx}
\usepackage{tikz}
\usepackage{todonotes}
\usepackage{dsfont}

\providecommand{\ms}{\mathsf}

\providecommand{\mc}{\mathcal}
\def \CC{\omega}

\providecommand{\N}{\mathbb{N}}
\providecommand{\R}{\mathbb{R}}

\providecommand{\Z}{\mathbb{Z}}
\providecommand{\T}{S}

\providecommand{\D}{D}

\providecommand{\W}{\mathbb{W}}

\renewcommand{\P}{\mathbb P}
\providecommand{\Pc}{\P^{\ms{Coup}}}
\providecommand{\Pca}{\P^{\ms I}}
\providecommand{\Pcb}{\P^{\ms D}}
\providecommand{\Fc}{F^{\ms{Scr} }}
\providecommand{\Ec}{\E^{\ms{Coup}}}

\providecommand{\E}{\mathbb E}

\providecommand{\Cco}{\omega^{\ms B}}

\providecommand{\scr}{\ms{scr}}

\providecommand{\Cs}{\CC^{\scr}}
\providecommand{\Css}{\CC^{\scr}}

\renewcommand{\H}{\mc H}
\providecommand{\tH}{\widetilde H}

\providecommand{\wt}{\widetilde}
\renewcommand{\d}{{\rm d}}
\providecommand{\PBB}{\P_{\ms{BB},N}}
\providecommand{\BB}{\ms{BB}}
\providecommand{\dd}{{\rm d}}
\providecommand{\mun}{\mu_{N, \b}}

\providecommand{\es}{\emptyset}
\providecommand{\one}{\mathbbmss{1}}

\providecommand{\Ps}{\mc P_{\ms{s}}}
\providecommand{\Pu}{\mc P}
\providecommand{\Ii}{I^{(i)}}

\providecommand{\Es}{E^{\scr}}
\providecommand{\EB}{E^{\ms B}}

\providecommand{\Pst}{\mc P_s}
\providecommand{\ent}{{\ms{ent}}}

\providecommand{\Pois}{\ms{Pois}}

\providecommand{\Conf}{\ms{Conf}}
\providecommand{\Neut}{\ms{Neut}}

\renewcommand{\div}{\ms{div}\,}

\providecommand{\Emp}{\ms{Emp}_N}

\providecommand{\DA}{\ms{Avg}_R}

\providecommand{\vgN}{\CC}%{\vec{\gamma}_N}
\providecommand{\gN}{g}
\providecommand{\vzN}{\vec{z}_N}

\providecommand{\WR}{\frac1R \tH_{\Lambda_R}}
\providecommand{\WRM}{\W^{M}_R}

\providecommand{\WMM}{\W^M}

\providecommand{\rpp}{x_-}
\providecommand{\rmm}{x_+}
\providecommand{\cm}{c_{\ms{mov}}}

\renewcommand{\t}{\tilde}

\providecommand{\e}{\varepsilon}
\renewcommand{\b}{\beta}
\providecommand{\z}{\zeta}

\providecommand{\de}{\delta}

\def\tcb{}%\textcolor{blue}}

\def\bi{\begin{itemize}}
\def\ei{\end{itemize}}

\def\be{\begin{equation}}
\def\ee{\end{equation}}
\def\bea{\begin{eqnarray}}
\def\eea{\end{eqnarray}}
\def\nn{\nonumber}
\def\ff{\infty}
\def\({\left(}
\def\){\right)}
\def\[{\left[}
\def\]{\right]}

\def \P{\mathbb P}
\def \E{\mathbb E}

\def \CP{\CC^{\ms P}}
\def \RP{R^{\ms P}}
\def \ZB{Z^{\ms B}}
\def \ZP{Z^{\ms P}}
\def \BP{B^{\ms P}}

\def \II{\mathcal{I}}

\def\La{\Lambda}
\def\taub{\tau_{\mathcal L_b}}

\def\Th{\Theta}
\def\a{\alpha}
\def\th{\theta}
\def\ol{\overline}
\def\ti{\times}
\def\vp{\varphi}

\def\su{\subseteq}

\def\been{\begin{enumerate}}
\def\enen{\end{enumerate}}
\def\beit{\begin{itemize}}
\def\enit{\end{itemize}}
\def\im{\item}
	\def\pa{\partial}

	\newtheorem{theorem}{Theorem}[section]
\newtheorem{corollary}[theorem]{Corollary}

\newtheorem{lemma}[theorem]{Lemma}
\newtheorem{proposition}[theorem]{Proposition}
\theoremstyle{definition}
\newtheorem{definition}[theorem]{Definition}

\newtheorem{remark}[theorem]{Remark}

	\usepackage{etoolbox}
	\patchcmd{\section}{\scshape}{\bfseries}{}{}
	\makeatletter
	\renewcommand{\@secnumfont}{\bfseries}
	\makeatother

	\usepackage[normalem]{ulem}	
	\usepackage[section]{placeins}


\begin{document}
		
		\author{Christian Hirsch}
		\author{Sabine Jansen}
		\author{Paul Jung}
		\address[Christian Hirsch]{Department of Mathematics, Aarhus University,  Ny Munkegade 118,  8000 Aarhus C, Denmark}
		\email{hirsch@math.aau.dk}
		\address[Sabine Jansen]{Mathematisches Institut, Ludwig-Maximilians-Universit\"at M\"unchen, Theresienstra\ss e 39, 80333 M\"unchen, Germany}
		\email{jansen@math.lmu.de}
		\address[Paul Jung]{Department of Mathematical Sciences, KAIST, 291 Daehak-ro Yuseong-gu Daejeon 34141 South Korea}
		\email{pauljung@kaist.ac.kr}
		
		\thanks{PJ was funded in part by the National Research Foundation of Korea grants NRF-2017R1A2B2001952 and NRF-2019R1A5A1028324.} 
		
		\title{Large deviations in the quantum quasi-1D jellium}
		
		\date{\today}
		
		\begin{abstract}
			Wigner's jellium  is a model for a gas of electrons. The model consists of $N$ particles with negative unit charge in a sea of neutralizing homogeneous positive charge spread out according to Lebesgue measure, and interactions are governed by the Coulomb potential. In this work we consider the quantum jellium on quasi-one-dimensional spaces with Maxwell-Boltzmann statistics. Using the Feynman-Kac representation, we replace particle locations with Brownian bridges. We then adapt the approach of Lebl\'e and Serfaty (2017) to prove a process-level large deviation principle
			for the empirical fields of the Brownian bridges.
			
			\vspace{2mm}

		\end{abstract}
		
		\maketitle

	\tableofcontents
	
	\section{Introduction}
	\label{intrSec}

	\begin{figure}[!htpb]
	\input{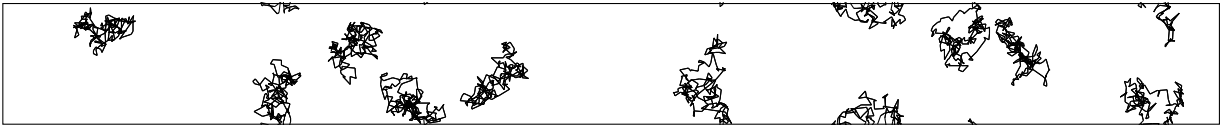}
	\caption{Illustration of the Brownian-bridge interpretation of the quasi-one-dimensional quantum jellium}
	\label{modelFig}
\end{figure}

	The jellium, introduced by Wigner \cite{wigner1934interaction}, is a model for a gas of electrons moving in a (locally) uniformly smeared neutralizing background of positive charge. 
	Wigner predicted that when the potential energy of the system overwhelms the kinetic energy, the electrons would form a ``close-packed lattice configuration''.
	We are interested in the quasi one-dimensional quantum jellium which is a model of electrons inside an insulated conducting wire with some thickness (quasi one-dimensional systems are infinite in one direction and bounded in all other directions, e.g. an infinite cylinder).

	In the physics literature, Deshpande and Bockrath \cite{deshpande2008one} observed Wigner-crystal type behavior in experiments on carbon nanotubes (see also \cite{meyer2008wigner, deshpande2010electron}) and \cite{PhysRevB.94.115417} studied thermal
		effects on crystallization.
	In the mathematics literature, the classical quasi-one-dimensional jellium, locally two-dimensional, was considered by \cite{choquard1983two, forrester1983two} where the system was seen to be exactly solvable when the inverse temperature satisfies $\b = 2$ (with electrons having unit charge). An expansion of the free energy of the system, in terms of the period of a hypercube in the bounded directions, is found in \cite{forrester1991finite}. Translation symmetry breaking was shown (in support of Wigner's prediction) when $\b$ is an even integer in \cite{vsamaj2004translation, jansen2008laughlin} where the latter article focused on a connection to Laughlin states in cylindrical geometry. This symmetry breaking was extended to all values of $\b$ in \cite{aizenman2010symmetry}.

	Wigner's original model was of course in the quantum setting. For the quantum one-dimensional jellium \cite{brascamp1975some} proved crystallization (or translation symmetry breaking), by introducing the well-known Brascamp-Lieb inequality, for the one-dimensional quantum system at sufficiently low temperatures. This was extended to all temperatures in \cite{jansen2014wigner}.
	
	In this work we investigate the free energy, in the thermodynamic limit, of the quantum quasi one-dimensional jellium at low density through the lens of large deviations. The large deviations approach to investigating Coulomb gases has recently been used by \cite{garcia2017large, berman2018large, chafai2020macroscopic, vaios, liu2020large} among others. At the level of fluctuations or level three (process level) large deviations, results have been obtained by
	 \cite{serfInv, serfBook} for $d\ge 2$ (in one dimension, they employ the log-potential which is not the Coulomb potential for that dimension, but allows for connections to random matrix theory).
	In order to extend this approach to the quantum setting, we use the Feynman-Kac representation and replace particle locations with Brownian bridges or Brownian loops which are sometimes referred to as filaments. {One may equivalently take the view that the particle locations are `marked' with Brownian bridges.} In the context of Coulomb gases this representation goes back to \cite{ginibre1965reduced} and was used by \cite{brascamp1975some, jansen2014wigner} (see also \cite{brydges1999coulomb}). The introduction of Brownian bridge filaments into the picture of the Gibbs measures complicates the mathematics considerably, and to our knowledge, no such rigorous study for the thermodynamic limit of the quantum jellium has yet been established beyond the one-dimensional setting.
	
	{While the model in this work} is geared towards the free energy of the jellium in the quantum setting, most of the analysis below can be adapted to other random marked point configurations. Such point configurations, with marks in path spaces, come up naturally in the study of infinite-dimensional interacting diffusions~\cite{lang1977,fritz87gradient,spohn1986}, see  \cite{deuschel87diffusions,dereudre2003interacting-brownian}. A Gibbs variational principle for interacting diffusions with infinite time-horizon, seen as space-time Gibbs random fields, was proven by Dai Pra, Roelly, and Zessin~\cite{daipra-roelly-zessin2002}. The aforementioned references focus on diffusions interacting via superstable pair potentials, but singular interactions---including interactions of the Coulomb type---have been studied as well, in part motivated by random matrices and Dyson's model \cite{spohn1986dyson,osada2012,tsai2016}. {The large deviation techniques explored here could be of relevance to diffusions with singular interactions.}

	One complication that arises in the Feynman-Kac representation using Brownian bridge filaments, is how to deal with the Fermi-Dirac statistics inherent in the representation. For one-dimensional systems, this is easy since the Brownian reflection principle allows such statistics to be equated with the consideration of non-intersecting Brownian bridges (see \cite[Section 3.1]{jansen2014wigner} for details).
	In higher dimensions assuming the absence of hard cores, particles do not collide, thus no such simplification is possible. Since the quasi-one-dimensional jellium is locally of higher dimension than one, we will simplify things by considering Boltzmann statistics. Such a simplification is reasonable at low density, when the Debye length is much smaller than the interparticle distance, as noted in \cite{brydges1994absence, brydges1999coulomb}. In this work, we will take as a starting point, the so-called ``magic formula'' (see Section V.1 of \cite{brydges1999coulomb}) which in our case amounts to replacing electrons with Brownian bridges of time-length $\b$. 

	{In Sections \ref{sec:truncation}-\ref{sec:energy}, we introduce the model in detail. In Section \ref{sec:main result}, we state our main result  establishing a (level three) large deviation principle (LDP) for Brownian bridge configurations in the quasi-one-dimensional jellium, and in Section \ref{sec:contributions}, we summarize our main contributions. In Section \ref{outlineSec}, we give a top level proof of our LDP. The remaining Sections \ref{sec:energy entropy}-\ref{sec:quasicontinuity} concern the technically most challenging part of our result-- establishing the so-called quasi-continuity of the specific energy which leads to the lower bound of our LDP.}

	\section{Model definition and main results}
	\label{modSec}
\tcb{Let $\D:=[0, 1]^k$ be the unit cube in dimension $k\ge1$.}	The {\it jellium} describes $N\ge1$ negatively-charged particles in the finite quasi one-dimensional domain 
	\begin{align}\label{def:TN}
		\La_N := [0, N] \ti \D 
	\end{align}
	of positive background charge, where particles are subject to pairwise interaction by the \tcb{$(1 + k)$-dimensional} Coulomb potential. Since $|\D| = 1$, the system is charge neutral upon equipping $\La_N$ with a 
	positive smeared background charge with density one \cite{wigner1934interaction}. The potential energy of a collection of $N\ge1$ particles
	$$\vzN = (z_1, \dots, z_N) = ((x_1, y_1), \dots, (x_N, y_N)) \in \La_N^N$$ 
	is given by
	\begin{align}
	\label{potEnEq}
	 \H_{\La_N}(\vzN) = \sum_{i < j} g(z_i - z_j) - \sum_{i \le N} \int_{\La_N} \gN(z - z_i) \d z + \frac12 \iint_{\La_N^2} \gN(z - z') \d z \d z' 
	\end{align}
	where $g(z)$ is the Green's function \tcb{of the Laplacian} on the domain $$\T := \R \ti \D.$$ That is, \tcb{writing $o \in S$ for the origin}, 
	$-\Delta g = \de_o$,	where we use free boundary conditions in the $x$-direction and periodic boundary conditions in the $y$-directions. The Green's function admits the eigenfunction expansion (see \cite[Sec. 2]{aizenman2010symmetry})
	\begin{eqnarray} \label{kernel_expansion}
	g(z-z') & = &
	-|x-x'|/2 \ + \  \sum_{n\ge 1}  (2\sqrt{\lambda_n})^{-1}\, e^{-|x-x'| \sqrt{\lambda_n}
	}  \, \ol{\vp_n(y)}\vp_n(y')
	\end{eqnarray}
	in terms of the eigenfunctions $-\Delta_D \vp_n(y) = \lambda_n
	\vp_n(y)$.

	\subsection{Short-distance regularization}\label{sec:truncation}

	%%%SMEARING
	The mathematical interest in Coulomb systems stems from their long-range interactions. Besides decaying slowly over long distances, the Coulomb potential also exhibits a singularity at the origin, thereby creating additional technical difficulties.  We circumvent the issues arising from the singularity by
	following \cite{serfInv} and 
	using Onsager's approach of smearing out point charges uniformly in a sphere $\pa B_\eta(z)$ of a fixed small radius $\eta < 1/4$ around a particle located at $z \in \La_N$ \cite{onsager}. With regards to the smearing, we view smeared particles to be in the larger space $S$ so that there is no ambiguity for how to smear $z$'s with $x$-coordinates near the boundaries of $[0, N]$. We write 
	$$
		\de^\eta_z := \frac1{|\pa B_\eta(z)|} \int_{\pa B_\eta(z)} \de_{z'} \d z'
	$$
	for the measure of the smeared point charge. This operation replaces the potential energy $\H_{\La_N}(\vzN)$ by 
	\begin{align}
	\label{potEnEetaEq}
	\H_{\La_N}^\eta(\vzN) & = \sum_{i < j} g^{\eta, \eta}(z_i - z_j) - \sum_{i \le N}\int_\T \int_{\La_N} \gN(z - z' )\d z\de^\eta_{z_i}(dz')
	+ \frac12 \iint_{\La_N^2} g(z - z') \d z \d z'.
	\end{align}
	where 
	$$g^{\eta, \eta}(z) := \iint_{\T^2}g(z' - z'') \de_z^\eta(\d z') \de_o^\eta(\d z'').$$
	{The function $\H_{\La_N}^\eta(\vzN)$ is invariant with respect to permutation of the arguments $z_1,\dots, z_N$; by some abuse of notation we use the same letter $\H_{\La_N}^\eta$ for the function of the \emph{set} $\{z_1,\dots, z_N\}$, with the $z_i$'s pairwise distinct.}

{ We note that in \cite{serfInv}, the $\eta$-smearing was eventually removed by taking $\eta \to 0$ in the thermodynamic limit. However, the techniques from \cite{serfInv} do not extend smoothly to the quantum setting. Indeed, their central step \cite[Lemma 5.9]{serfInv} consists of a comparison to a regularized configuration, where particles that are too close to each other are carefully separated.  In the setting of Brownian bridges, such regularizations would differ at different times, and it is not clear whether this can be achieved at low entropic cost while simultaneously satisfying the continuity restrictions induced by the Brownian bridge paths. For instance, a significant technical difficulty arises when one Brownian particle enters the $2\eta$-sphere of another particle, and exits on the opposite side, nowhere close to its entry point. \footnote{{The naive approach of replacing exceptional bridges by near-straight line paths  (approximately constant in time, with small fluctuations around their starting points) seems difficult. Indeed, such a replacement would lead to major changes of the interaction potential with other bridges, which seem challenging to control. The sort of changes of the interaction potential resulting from any short-range regularization involving near-straight line paths can be seen in Proposition \ref{prop:detscreening}. However, unlike the situation there, changes to configurations due to any short-range regularization would not be restricted to bridges near the boundaries of the region.}
	}}

	\subsection{Configuration space} 
{When thinking of particles as electrons, the quantum setting employs Fermi-Dirac statistics as in \cite{jansen2014wigner}. However, the standard one-dimensional method of employing the Brownian reflection principle, as in \cite{jansen2014wigner}, does not extend to the quasi-1D setting. We therefore work with the simpler Boltzmann statistics, which, as mentioned already, becomes a reasonable approximation at low densities \cite{brydges1999coulomb}.}

	Invoking the Feynman-Kac representation \cite{brydges1999coulomb, ginibre}, the Boltzmann statistics for the quantum jellium can be deduced by analyzing a system of standard Brownian bridges subject to Coulomb interactions in each time slice. 

	More precisely, let $C([0, \b])$ be the space of continuous paths from $[0, \b]$ to $\T$, with times $0$ and $\b$ identified {i.e., $\omega(0)= \omega(\beta)$}, and let $\Conf$ denote the space of all configurations of the form
	\be \label{confidef}
		\CC = \bigcup_{j \in J}{\{}(b_j(0), b_j) {\}}
	\ee 
	with $J$ a countable set {(possibly finite)} and $(b_j)_{j\in J}$ a {locally finite} collection of bridges, i.e. elements in $C([0, \b])$, such that for every compact set $K\su S$, the set $K\cap \{b_j(0):\, j\in J\}$ is finite.  {The space  $C([0,\beta])$ is equipped with the supremum norm and associated Borel $\sigma$-algebra, and $S\times C([0,\beta])$ with the product $\sigma$-algebra.} The $\sigma$-algebra of $\Conf$ is the smallest one such that the maps 
	$$\CC \mapsto \#\{j\in J:\,(b_j(0), b_j) \in T\}$$
	are measurable for all measurable \tcb{$T \su S\ti C([0, \b])$} ($\#$ denotes cardinality). In particular, by shifting the marks in the sense of $\{(b_j(0), b_j - b_j(0))\}_{j \in J}$, our framework corresponds to that of marked point processes as presented in \cite{georgii2}. Let $\Pu$ be the space of probability measures on $\Conf$.

	The projection $\pi_\La:\,\Conf\to \Conf$ to a measurable set $\La \su S$ maps the configuration $\CC$ given by \eqref{confidef} to 
	\be\label{def:proj}
		\pi_\La(\CC) = \omega_\La := \big\{(b_j(0), b_j)\in\CC:\, b_j(0)\in\La\big\}.
	\ee
 Further, let 
\be \label{projdef}
	P_\La = P\circ \pi_\La^{-1}
\ee
be the image of $P\in \Pu$ under $\pi_\La$ and let 
	\begin{align}\label{def:conf}
		\Conf(\La) & := \pi_\La(\Conf)
	\end{align} 
denote the space of configurations {with starting points} in $\La$.

We will typically use the shorthand $b\in \CC$, or similarly, suppress the first coordinate of the ordered pair $(b(0),b)$ since this can be easily found by evaluating $b$ at time $0$. {We also use the notation 
\be
	\omega(t) := \{ b(t):\ b \in \omega\}\subset S,\quad t\in [0,\beta].
\ee
}

\subsection{Brownian bridge measure}
%%GIBBS MEASURE
{In the following $\PBB(\d b)$ is the distribution of a standard Brownian bridge with uniformly distributed starting point in the background $\La_N$. Note however, that the bridge is allowed to leave the region $\La_N$. Denote by $\P_N$ the image of the product measure 
\begin{align}\label{def:BBmeasure}
	 \PBB(\d b_1) \otimes \cdots \otimes \PBB(\d b_N)
\end{align} 
under the mapping $C([0,\beta])^N\to \mathsf{Conf}(\Lambda_N)$, $(b_1,\dots, b_N)\mapsto \{b_i:\ i=1,\dots, N\}$. Equivalently, $\P_N$ is the distribution of a Binomial point process with independent Brownian bridge marks.}

Henceforth, we analyze the Gibbs measure
\begin{align}\label{def:muNb}
	\mun(\d\CC) := \frac1{Z_{N, \b}} \exp\Big(- \int_0^\b \H_{\La_N}^\eta(\CC(t)) \d t\Big) \P_N(\d\CC), 
\end{align}
with partition function
$$
	Z_{N, \b} = {\int_{\mathsf{Conf}(\Lambda_N)}} \exp\Big(-\int_0^\b \H_{\La_N}^\eta(\CC(t)) \d t\Big) \P_N(\d \CC).
$$
In other words, $ \mun(\d\CC)$ is absolutely continuous with respect to $\P_N(\d \CC)$ with Radon-Nikodym density
$$\frac1{Z_{N, \b}} \exp\Big(- \int_0^\b \H_{\La_N}^\eta(\CC(t)) \d t\Big).$$

\subsection{Empirical field}
 Discrete shifts by $k \in \Z$ act on $(x, y) \in S$ and $b \in C([0, \b])$ as follows:
\be\label{def:shift}
	 \th_k (x, y) = (x - k, y), \quad (\th_k b)(t) = \th_k(b(t)).
\ee 
\tcb{The main result of this paper is an LDP on the level of empirical fields, sometimes referred to as a level-three or process-level LDP. To put this into a general context, note that the paradigm of statistical physics is to consider large systems of particles and describe the behavior of observables that are realized as an average over such a system \cite[Ch. 6]{rassoul2015course}. More precisely, }for a collection of bridges $\CC$, we consider the \emph{empirical field}
\begin{align}\label{def:EMP}
\Emp(\CC) := \frac 1N \sum_{0 \le i \le N - 1} \de_{\th_i \CC}, 
\end{align}
as an element of the space $\Pu$ of probability measures on $\Conf$. Our choice of averaging only over integer shifts reflects the occurrence of crystallization and the appearance of fractional charges as striking characteristics of quasi one-dimensional Coulomb systems \cite{aizenman2010symmetry, aizenman1980structure}.
{We remark that since $\CC$ is a configuration consisting only of bridges, i.e. not including any background charge, the $\th_k$-shifts do not affect the background charge.}

\tcb{We would like to consider the limiting behavior of observables on $N$-particle systems as $N$ grows, thus we next describe the relevant topology for these limits.}

%%%%%%%%%%%%%%%%%%%%%%%%%%%%%%%%%%%%%%%%%%%
%%%%%% TOPOLOGY
%%%%%%%%%%%%%%%%%%%%%%%%%%%%%%%%%%%%%%%%%%%

\subsection{Topology {and $\sigma$-algebra on $\mc P$}}\label{sec:topology}
We now specify the topology on $\mc P$ entering the large deviation principle for the distribution of the empirical fields.
A bounded measurable function $f:\Conf\to \R$ is \emph{local} -- in symbols $f \in \mc L_b$ -- if it depends only on bridges with starting points in some bounded region, i.e., if there exists a bounded set $\La\su S$ such that $f(\CC) = f(\CC_\La)$ for all $\CC$. The $\taub$-topology is the smallest topology such that all evaluation maps $P \mapsto \E_P[f(\CC)]$, $f\in \mc L_b$, are continuous. We endow $\Pu$ with the smallest $\sigma$-algebra such that for every bounded or non-negative measurable $f:\Conf\to \R$, the map $P\mapsto \E_P[ f(\CC)]$ is measurable as well \cite{georgii2}.

Since the $\taub$-topology is neither metrizable nor separable, many arguments become more technical than their analogs in the weak topology of \cite{serfInv}. The latter is weaker than the $\taub$-topology since it relies on bounded local test functions that are continuous on $\Conf$ where point configurations in $\Conf$ are identified with counting measures and equipped with the vague topology. However, we found the restriction to continuous test functions to be not versatile enough when dealing with configurations of bridges instead of particles. For example, the upper-semicontinuity argument in \cite[Lemma 5.7]{serfInv} uses that in a certain infimum over vector fields a minimum is attained. This is not clear if the fields vary over time. 

 In particular, the $\sigma$-algebra and topology on $\mc P$ are the analogs of the cylinder $\sigma$-algebra $\mc B^\mathrm{cy}$ and the $\tau$-topology described in the context of Sanov's theorem in \cite[Chapter 6.2]{dz98}, see the comments after \cite[Definition 0.2]{georgii1}. One should keep in mind that the $\sigma$-algebra on $\mc P$ is not the Borel-$\sigma$-algebra of the $\taub$-topology and so, in particular, not every open set is measurable. As a consequence, large deviation principles must be formulated carefully, and compactness of level sets of rate functions no longer implies exponential tightness, see \cite{eichelsbacher-grunwald99} and the remarks preceding \cite[Lemma 1.2.18]{dz98}.

\begin{remark}
	For point processes with finite intensity, another choice is the class $\mc L$ of \emph{tame} local functions and the associated $\tau_\mc L$-topology \cite{georgii2}. For point processes without marks, a local function is \emph{tame} if $|f(\CC)|\le C(1+\mc \#\CC_\La)$ for some $C>0$, some bounded $\La$, and all $\CC$. The class of tame local functions is strictly larger than the class of bounded local functions, therefore the topology $\tau_\mc L$ is finer than the topology $\taub$, however as noted in the proof of Proposition 2.6 in \cite{georgii2}, on sublevel sets $\{\ent(P)\le c\}$, $c>0$, the topologies coincide. 
\end{remark}

{In view of the above discussion, it is helpful to single out the measurable open neighborhoods around some $P\in\mc P$. We use the notation
	\begin{align}\label{def:N(P)}
		\text{ $\mc U_{\text{meas}}(P)$ is the family of measurable neighborhoods of $P$}.
	\end{align}
	The set $\mc U_{\text{meas}}(P)$ contains, in particular, a basis of the $\taub$-topology of cylinder sets 
	\be \label{topology-basis}
	U_{F, P} := \bigcap_{j \le n} \big\{Q \in \mc P:\, \bigl|\E_Q[f_j]- \E_P[f_j]\bigr| < \de \big\}
	\ee
	with $\de>0$, $n \ge 1$, and $F = \{f_1, \dots, f_n\}$ a finite set of bounded local functions.}

%%%%%%%%%%%%%%%%%%%%%%%%%%%%%%%%%%%%%%%%%%%
%%%%%%ENTROPY
%%%%%%%%%%%%%%%%%%%%%%%%%%%%%%%%%%%%%%%%%%%
\subsection{Entropy} \label{subsec:entropy}
A \emph{bridge configuration} refers to a deterministic locally finite collection of $C([0, \b])$-marked points in $\T$, whereas the space $\Ps$ consists of shift-invariant probability measures ($P\circ \th_k = P$ for all $k\in \Z$) describing \emph{stationary bridge processes}. 

We let $\Pois$ denote the law of a unit-intensity homogeneous Poisson point process on $\T$, marked  with {independent} Brownian bridges of diffusion parameter 1 and time-length $\b$. {See~\cite[Chapter~5.3]{last2017lectures} for an abstract definition of independent markings.} 
 For $P \in \Ps$, let
\begin{align}\label{def:entP}
	\ent(P) := \lim_{N \to \ff} \frac1N \int \log\Big(\frac{\d P_{\La_N}}{\d \Pois_{\La_N}}(\CC)\Big) P_{\La_N}(\d \CC), 
\end{align}
denote the \emph{specific relative entropy} of $P$ with respect to $\Pois$ \cite{georgii2}, applying the convention that $\ent(P) = \ff$ if for some $N$, $P{_{\La_N}}$ is not absolutely continuous with respect to ${\Pois_{\La_N}}$. It can be seen from \tcb{sub-additivity arguments in} \cite{georgii1, georgii2, rassoul2015course} that this limit exists. In order to state the LDP formally in the space $\Pu$ instead of $\Ps$, we set $\ent(P) = \ff$ if $P \in \Pu \setminus \Ps$.

%%%%%%%%%%%%%%%%%%%%%%%%%%%%%%%%%%%%%%%%%%%
%%%%%%ENERGY
%%%%%%%%%%%%%%%%%%%%%%%%%%%%%%%%%%%%%%%%%%%
\subsection{Energy}\label{sec:energy}
A crucial property of Coulomb systems is an intimate relation between the energy \eqref{potEnEq}, and the energy of an associated electric field (for instance, this appears 
in the \emph{splitting formula} in \cite[Proposition 3.3]{serfBook}). More precisely, consider a bounded region {$\La$} and a configuration $\CC \in \Conf(\La)$. For each time $t \le \b$, the points $b(t)$ of the Brownian bridges $b\in \CC$ together with the homogeneous opposing background charge in $\La$ create the electrostatic potential 
\begin{align}\label{elFieldEq}
	 V_t(z, \CC, \La) := \int_\T \gN(z - z') \Big(\sum_{b\in {\CC_\La}} \de^\eta_{b(t)} - \one_\La \Big)(\d z').
\end{align}
The integral is over $S$ rather than $\La$ because Brownian bridges $b(t)$ may leave $\La$.  Let 
\be \label{Hdef}
	H_{\b, \La}(\CC) := \frac1{2\b}\int_{\T \ti [0, \b]}|\nabla V_t(z, \CC, \La)|^2 \d(z, t).
\ee
Denote by 
\begin{align}\label{def:neut}
\Neut(\La) & := \big\{\CC\in \Conf(\La):\, \#\CC = |\La| \big\},
\end{align} 
the space of all \emph{charge-neutral} configurations.
In Appendix \ref{sec:aux}, we integrate by parts to prove the following variation of a standard identity.
\begin{lemma}[Energy in terms of electric field]
	\label{splitLem}
	Let $\La = {[L_-, L_+] \ti D}$ with $L_-, L_+\in \Z$ and let $\CC \in \Neut(\La)$. Then,
	\begin{align}
		\label{splitEq}
		\frac1\b \int_0^\b\H^\eta_\La(\CC(t)) \d t = H_{\b, \La}(\CC) -N g^{\eta, \eta}(o).
	\end{align}
\end{lemma}

Relying on Lemma \ref{splitLem}, the expected energy per unit length along the $x$-direction for a stationary bridge process $P\in \Ps$ is now defined in two steps. Proceeding na\"ively, one could define the energy content of a finite window $\La_N$ by minimizing over all possible boundary conditions and then taking the expectation and the limit over an increasing window size. Minimizing over boundary conditions is not necessarily natural from a physics point of view but it turns out to be technically convenient. In statistical mechanics it is customary to impose conditions like temperedness on boundary conditions~\cite{ruelle70} that give a handle on interactions between a bounded window $\Lambda_N$ and the outside $\Lambda_N^c$. In our context, the interactions we need to control are caused by bridges starting outside $\Lambda_N$ but invading $\Lambda_N$ deep inside. We impose the following conditions. For an interval $I = [L_-, L_+] \su \R$, we first define the {\it eroded} interval
$$I^- := \big[L_- + \big\lceil|I|^{7/8}\big\rceil, L_+ - \big\lceil |I|^{7/8} \big\rceil\big]$$
 and for $K = I \ti D$ we put $K^- = I^- \ti D$. Then, 
\be \label{eq:thetadef}
 \Th(K) := \big\{\CC\in\Conf:\, b\in \CC \text{ and }b(0)\not\in K\text{ implies }b \cap K^- = \es\big\}, 
\ee
{with $b\cap K^- = \emptyset$ a shorthand for $\{b(t):\ t\in [0,\beta]\}\cap K^- = \emptyset$,}
consists of all configurations $\CC \in \Conf$ for which any bridge in $\CC$ with starting point outside of $K$, stays outside of $K^-$ for all times.

{\begin{remark}[Choice of exponents]\label{rem:exponent choice}
	The exact choice of the exponent $7/8$ above is not important. We have chosen this value to allow for the proof of Proposition \ref{rateFunProp}, see Eq. \eqref{eq:exponent choice}, as well as for reasonable choices of parameters later, see Remark \ref{rem:reasonable choices}. The exponents in Lemma \ref{lem:entropy-estim} below are also chosen with Eq. \eqref{eq:exponent choice} in mind. 
\end{remark}}

In the following, $\E_P$ denotes the expectation with respect to $P \in \Pu$.

\begin{definition}[Expected specific energy]
	\label{specEnDef}
The \emph{expected specific energy} of $P\in \Ps$ {with $\mathsf{ent}(P)<\infty$} is {defined, whenever this limit exists,} as
	\begin{equation}
		\label{wDef}
		\W_\b(P) := \lim_{M \to \ff}\W_\b^M(P),
	\end{equation}
where 
	\begin{equation}
		\label{wDef1}
		\W_\b^M(P) := \lim_{N \to \ff} \, \E_P\big[(N^{-1}\tH_{\b, \La_N}(\CC)) \wedge M\big], 
	\end{equation}
	 and for $\CC\in \Conf$ and $K\su S$,
	\be \label{eq:wkdef}
			\tH_{\b, K}(\CC) := \inf_{\substack{\La \supset K, \tilde \CC\in \Neut(\La) \\ \t\CC\in \Th(K),\, {\t\CC_K = \CC_K}}} \Biggl\{\frac1{2\b} \int_{K\ti [0, \b]} \bigl|\nabla V_t(z, \t\CC, \La)|^2 \d(z, t) \, \Biggr\}.
	\ee
	The infimum considers only domains $\La$ of the form $\La = I \ti D$ for an interval $I \su \R$ with integer endpoints.
\end{definition}
\begin{remark}
By approximating arbitrary Brownian bridges by piecewise linear functions with rational interpolation points, we may replace the infimum in \eqref{eq:wkdef} by the infimum over a countable set. In particular, there are no issues with the measurability of  $\omega \mapsto \tilde H_{\beta, K}(\omega)$ when forming expectations. 
\end{remark}

Let us relate the definition of $\W_\b(P)$ to the analogous definition in \cite{serfInv}. Firstly, as mentioned above, we found it necessary due to the presence of highly fluctuating Brownian bridges to impose the boundary condition $\Th(K)$. Moreover, in \eqref{wDef} the order of forming the expectation and taking the limit is reversed. When working with bridges instead of just their endpoints, we found it difficult to control the asymptotic behavior of the specific energy of configurations in large volumes on the level of realizations. Loosely speaking, after forming expectations the effects of the fluctuation are less severe and a large-volume analysis becomes tractable.

%%%%%%%%%%%%%%%%%%%%%%%%%%%%%%%%%%%%%%%%%%%
%%%%%%LDP
%%%%%%%%%%%%%%%%%%%%%%%%%%%%%%%%%%%%%%%%%%%
\subsection{Main result: large deviation principle}\label{sec:main result}
Defining the \emph{free energy} as
\begin{align}\label{free energy}
	\mc F_\b(P) := \W_\b(P) + \b^{-1} \ent(P)
\end{align}
{if $P\in \mathcal P_\mathrm s$ and $\mathsf{ent}(P)<\infty$, and $\mathcal F_\beta (P)=\infty$ otherwise},
we now state the large deviation principle, where $\ol A$ and $A^\circ$ refer to the closure and interior with respect to the $\taub$-topology of a measurable set $A \su \Pu$.

\begin{theorem}[LDP for the empirical fields]
	\label{mainThm}
	The empirical fields {$\{\Emp(\CC)\}_{N \ge 1}$} under $\{\mun\}_{N \ge 1}$ satisfy the LDP at speed $N$ in the $\taub$-topology on $\Pu$ with {good rate function} 
	$$P \mapsto \b\mc F_\b(P) - \inf_{P' \in \Pu} \b\mc F_\b(P').$$
	That is, {for every measurable set $A\su \mc P$}, 
	$$\limsup_{N \to \ff}{\frac1N} \log \mun({\{{\CC_{\La_N}}:\,\Emp({\CC_{\La_N}}) \in A\}}) \le -\inf_{P \in {\ol A}}\big(\b\mc F_\b(P) - \inf_{P' \in \Pu} \b\mc F_\b(P')\big), $$
	and
	$$\liminf_{N \to \ff}{\frac1N} \log \mun({\{{\CC_{\La_N}}:\,\Emp({\CC_{\La_N}}) \in A \}}) \ge -\inf_{P \in {A^\circ}}\big(\b\mc F_\b(P) - \inf_{P' \in \Pu} \b\mc F_\b(P')\big)$$
\end{theorem}
\tcb{An important byproduct of our proof in Section \ref{outlineSec}, is that we are also able to describe the asymptotics of the free energy $\log \E_N\big[ \exp\big(- \b H_{\La_N}(\CC) \big)\big]$ for large $N$, where the expectation is with respect to $\P_N$ defined above \eqref{def:BBmeasure}.}	

Having formulated the LDP in terms of Brownian bridges, we now translate it back to the quantum setting under Maxwell-Boltzmann statistics. 	Essentially, this means specializing Theorem \ref{mainThm} to the starting points of the bridges. After stating the result, we expound on possible implications of this result from the perspective of mathematical physics. 

{
	In the operator-theoretic setting, the quantum-mechanical Hamiltonian for the quasi-1D jellium becomes
	$$H_{\La_N}^\eta = -\frac12 \sum_{i \le N} \Delta_{z_i} + \sum_{i < j} g^{\eta, \eta}(z_i - z_j) - \sum_{i \le N}\int_\T \int_{\La_N} \gN(z - z')\d z\de^\eta_{z_i}(dz')
	+ \frac12 \iint_{\La_N^2} g(z - z') \d z \d z',$$
	which is an operator on the Hilbert space $L^2(S^N)$. (Recall that we use free boundary conditions in the $x$-direction and periodic boundary conditions in the $y$-directions.) The expected value of an observable $\mathfrak{a}$  (self-adjoint operator in $L^2(S^N)$) in the finite-volume Gibbs state is
	\begin{equation} \label{eq:quantum-to-classical}
		\langle \mathfrak{a}\rangle_{N,\beta} = \frac{\mathrm{Tr}\, \mathfrak{a} \exp( - \beta H_{\Lambda_N}^\eta)}{\mathrm{Tr}\,\exp( - \beta H_{\Lambda_N}^\eta)}. 
	\end{equation} 
	The Feynman-Kac formula implies that expected values of observables that depend on particle positions $z_1,\dots, z_N$ (but not on momenta) are completely determined by the distribution of the initial points of the Brownian bridges: If $\mathfrak{a}$ is a multiplication operator with some function $G_N(z_1,\dots, z_N)$, then
	$$
		\bigl \langle G_N(z_1,\dots, z_N)\bigr \rangle_{N,\beta}  = \int G_N\bigl(  b_1(0), \dots, b_N(0)\bigr) \mu_{N,\beta} (\mathrm d \boldsymbol{b}). 
	$$ 
	Among the position-dependent observables of interest are extensive two-body quantities 
	$$
		G_N(z_1,\dots,z_N) = \sum_{1\leq i < j \leq N} h(z_j - z_i)
	$$	
	with bounded and compactly supported $h$. To make the connection with the stationary empirical field, notice that 
	\begin{equation}\label{eq:observable}
		G_N(z_1,\dots,z_N) = \sum_{k=0}^{N-1} u\bigl( \{z_i-k\}_{i=1,\dots,N}\bigr)
	\end{equation} 
	with 
	$$
		u\bigl( \{z_i\}_{i=1,\dots,N}\bigr)	= \frac12 \sum_{i=1}^N \one_{[0,1)\times D}(z_i) \sum_{j\neq i} h(z_j- z_i). 
	$$
	The function $u$ is bounded by a constant times the number of particles in $[0,1)\times D$ and it only depends on the configuration in some neighborhood of that set because $h$ is compactly supported. Thus, $G_N/N$ is the integral of the local and tame function $u$ against the stationary empirical field of the point configuration $\{z_1,\dots, z_N\}$. Theorem~\ref{mainThm}, Eq.~\eqref{eq:quantum-to-classical} and the contraction principle yield a large deviation principle for all macroscopic observables of that type.
	}
	
	\begin{corollary} \label{mainCor}
	{
		Let $u$ be a local and tame function of point configurations on the strip, and $G_N$ as in~\eqref{eq:observable}. Define $I_u:\R\to \R\cup \{\infty\}$ by
		$$
			I_u(x):= \inf\Bigl\{ \beta \mathcal F_\beta(P): \ P\in \Ps,\, \mc F_\beta(P) < \ff,\,\int_\Conf u\bigl( \{b(0):\, b\in \omega\}\bigr) P(\d \omega) = x \Bigr\}.
		$$
		Then $I_u$ is convex with compact sublevel sets and for all measurable subsets $A\su\R$, 
		\begin{align*} 
			\limsup_{N\to \infty} \frac1N \log \bigl \langle \one_{\{G_N/N \in  A\}}\bigr \rangle_{N,\beta} &\leq - \inf_{x\in \overline{A}} \bigl( I_u(x) - \min I_u\bigr), \\
			\liminf_{N\to \infty} \frac1N \log \bigl \langle \one_{\{G_N/N \in  A\}}\bigr \rangle_{N,\beta} &\geq - \inf_{x\in A^\circ} \bigl( I_u(x) - \min I_u\bigr).
		\end{align*}}
	\end{corollary} 
\tcb{Very loosely speaking, one consequence of Corollary \ref{mainCor} is that with a high probability, the $G_N/N$-observable is asymptotically close to one of the minimizers of $I_u$. The constraint appearing in the definition of $I_u(x)$ means that the expected value of the real-valued observable $u$ under the probability measure $P$ should be given by $x$.}
 
LDPs have a long history and deep connection with statistical physics; recent examples of this connection include \cite{serfInv,  garcia2017large, liu2020large} and in particular, in the quantum setting \cite{schlein}. LDPs of the type in this work are fairly standard in classical statistical mechanics with short-range interactions. Let us briefly review a few relevant aspects. LDPs play an important role in clarifying foundational questions in statistical mechanics.  Firstly, as explained already by Lanford \cite{lanford1973} for Gibbs measures with finite-range interactions, when the rate function $I_g$ has a unique minimizer, the LDP implies that the equilibrium distribution of the intensive observable $G_N/N$ is sharply peaked around its mean value, as it should be in light of the statistical mechanics approach for thermodynamics.
 Of course the hard work then consists in actually proving that $I_g$ has a unique minimizer, a question intimately tied to uniqueness of Gibbs measures and absence of phase transitions. 
Secondly, the LDP for the stationary empirical field is closely related to the Gibbs variational principle \cite{rassoul2015course} and it can be used to investigate the equivalence of ensembles on the level of states \cite{georgii1994equivalence}.

{On a related note,
the proof of Corollary \ref{mainCor} provides a variational representation of the limiting free energy. A consequence is the existence of the thermodynamic limit for our quasi-1D quantum systems. While such existence should follow from standard methods, to the best of our knowledge a 
proof for the systems considered here has not yet appeared in the literature. Moreover, while subadditivity arguments only yield the mere existence of the limit, the variational representation adds the conceptual insight that it results from balancing well-defined contributions from specific energy and specific entropy of an infinite system. The standard subadditivity arguments are also no longer available for inhomogeneous systems, and the work of \cite{serfInv} illustrates that the LDP approach is flexible enough to extend to such a setting.}

{
In the quantum setting, large deviations are less well understood. Large deviation estimates for particle numbers in the quantum ideal weakly interacting gas, and quantum gas with Boltzmann statistics and repulsive interaction at all densities and temperatures were studied by Lebowitz, Lenci and Spohn \cite{lebowitz-lenci-spohn2000} and Gallavotti, Lebowitz and Mastropietro \cite{gallavotti-lebowitz-mastropietro2002}.
	For quantum lattice systems, large deviations for broader classes of observables are available, see Ogata and Rey-Bellet \cite{ogata-reybellet2011} and the references therein. Subtleties of quantum large deviations are also discussed by Neto{\v{c}}n{\'y} and Redig \cite{netocny-redig2004} and some remarks in relation with the asymptotic equipartition theorem are made by de Roeck, Maes and Neto{\v{c}}n{\'y} \cite{deroeck-maes-netocny2006}. Quantum large deviations have also attracted interest in the context of quantum statistics and quantum information theory, see for example the quantum versions of Shannon-McMillan and Sanov theorems \cite{bjelakovic-etal2004inventiones,bjelakovic-etal2005sanov}.
}

{None of the aforementioned references on quantum large deviations deal with long-range interactions or Coulomb systems. The long-range nature of Coulomb interactions creates technical difficulties even for classical systems as demonstrated by \cite{serfInv}. Our results constitute a first step towards quantum large deviations for Coulomb systems.}

{ 
	Finally, from a more probabilistic point of view, while a law of large numbers and a central limit theorem capture the typical behavior and the typical fluctuations of a physical system, the LDP allows to describe the configurations under rare events. In particular, the distribution under a rare event becomes accessible via the Gibbs conditioning principle \cite[Theorem 7.3.3]{dz98} and importance sampling is also made possible, for an application see \cite{ray2018importance}.}

{One last comment is that, although Theorem \ref{mainThm} is stated for the Coulomb potential, many of the arguments remain valid more generally. The essential ingredients in our arguments are bounds on the number of bridges crossing various boundaries and bounds on the change in energy due to bridges deviating far from their starting points. For instance, the recent results obtained in \cite{dereudre} pave the way towards extending Theorem \ref{mainThm} to the 1D log-potential.}

\subsection{Summary of contributions.} \label{sec:contributions}{The goal of the present work is to provide a first example for which the fine-scale large deviation analysis achieved in the classical setting in \cite{serfInv} may be extended to the quantum setting via Feynman-Kac formalism. However, as demonstrated by our work here, passing from particles to bridges is far from just a formal step, but instead requires substantial novel arguments to address technical challenges. We now highlight the main novel contributions briefly, and provide a more detailed discussion in the corresponding subsections of the manuscript. 
\begin{itemize}
\item  When  using configurations of bridges instead of particles, one needs to use a stronger topology as discussed in Section \ref{sec:topology}. The natural topology in our setting, taken from \cite{georgii2}, is neither metrizable nor separable,  and thus the analog of various arguments in \cite{serfInv} become more technical.
	\im Related to the above, in the classical setting, a crucial ingredient in the proof of the LDP is a fine coercivity property from \cite[Lemma 4.2]{rouSe}. Loosely speaking, starting from point configurations and electric fields in bounded windows, the lemma allows to pass to a subsequence that converges to a deterministic limiting configuration in the entire domain. In the setting of bridges such a subsequence would need to be produced through the Arzel\`a-Ascoli theorem, which is difficult to apply due to the lack of equicontinuity of the bridges. \tcb{Although we are still using coercivity properties of the electric energy to derive energy bounds, these arguments do not give full compactness because of the extra bridge degree of freedom.}		Therefore, we develop a different route and rely on compactness results in the space $\mathcal P$ itself. 
\im The energy of configurations in the classical setting was defined by taking the infimum of all possible electric fields that are compatible with a given particle configuration \cite[Lemma 2.3]{serfInv}. One difficulty in our setting, is that in any given region of space,  Brownian bridges that start in that region, may leave and re-enter. Thus the total charge of the region is not even constant in time.  In particular, taking an unconstrained minimum would result in a pathological energy definition. In order to define the energy of bridge configurations, it is necessary to constrain the minimization problem with carefully chosen boundary conditions as we do in \eqref{eq:wkdef}.
\im The Brownian bridges could potentially be subject to extreme fluctuations, particularly in the setting of large deviations, and thereby prevent a screening construction that can be carried out consistently over the entire time horizon.  To deal with this issue, we introduce concepts of $\e$-dense abscissas (Definition~\ref{def:dense abs}) and $(M, \e)$-regular configurations (Definition~\ref{def:regularity}) and show that irregular configurations are so rare that they do not interfere with the screening construction. In particular, both Section \ref{sec:energy entropy} and Section \ref{subsec:boundary} are devoted to dealing with issues that arise due to the fluctuations of bridges, and they have no analog in previous works.
\im In Section \ref{sec:coupling}, we show that relying on the probabilistic concept of a coupling allows to shorten many of the volume computations from \cite{serfInv}.
\end{itemize}
}

\section{Top-level proof of the main result}
\label{outlineSec}

To prove Theorem \ref{mainThm}, we proceed in three steps. First, in Section \ref{sec:lowersemicont}, we show that the specific energy $\W_\b$ is lower semicontinuous. {Due to the quantum fluctuations introduced by the presence of Brownian bridges, the absence of super-additivity makes this step more difficult than one might expect.} Second, in Section \ref{upSec}, we establish the LDP upper bound. Third, in Section \ref{lowSec}, we outline how the LDP lower bound reduces to finding configurations of low energy. The construction of such configurations relies on the screening idea from \cite{serfInv, petrache, rouSe, ginzburg, sandSerf} and comprises the bulk of this paper, Sections \ref{sec:energy entropy}, \ref{sec:screening}, and \ref{sec:quasicontinuity}. In Appendix \ref{sec:aux}, we \tcb{prove a result expressing the energy of a configuration in terms of the electric field}. {A common motif throughout the manuscript is that the additional fluctuations coming from the Brownian bridges should not be seen as a simple model extension, since some of the arguments that were immediate in \cite{serfInv}  require  more attention and new ideas.}

\label{lscSec}
In the sequel, to simplify notation, we drop the $\b$-dependence and write
\begin{align*}
	H_\La(\CC) &\equiv H_{\b, \La}(\CC), &\tH_\La(\CC) &\equiv \tH_{\b, \La}(\CC),\\
	\mc \W(P) &\equiv \W_\b(P), &\W^M(P) &\equiv \W^M_\b(P),\\
	&\qquad\text{ and }&\mc F&\equiv \mc F_\b.
\end{align*} 

%%%%%%%%%%%%%%%%%%%%%%%%%
%%%%%%%%%%%%%%%%%%%%%%%%%
%%%SPLITTING FORMULA
%%%%%%%%%%%%%%%%%%%%%%%%%
%%%%%%%%%%%%%%%%%%%%%%%%%
\subsection{Lower semicontinuity}\label{sec:lowersemicont}

Our main results of this subsection are

\begin{proposition}[Existence of the energy density]
	\label{superAddLem}
	If $\ent(P) < \ff$, then the limit \eqref{wDef1} exists.
\end{proposition}

\begin{proposition}[Lower semicontinuity of free energy]
	\label{rateFunProp}
The free energy $\mc F$ in \eqref{free energy} is lower semicontinuous.
\end{proposition}
	{In order to prove these propositions, we will leverage that $\tH_{\La_N}$ over increasing regions $\{\La_N\}_N$ is `close' to being super-additive. 	The biggest obstacle to super-additivity is the possibility of configurations having many highly fluctuating Brownian bridges which cross the boundaries of the regions $\{\La_N\}$. \tcb{For a bridge $b = \bigl(b(t)\bigr)_{t \le \b}$, let us denote by $b_x(t)$ and $b_y(t)$, the spatial coordinates $x\in\R$ and $y\in D$, respectively, of the Brownian bridge at time $t$.}
We first present two lemmas which will make use of the following $x$-{\it range} functional of a bridge $b$ seen from its starting point:
\be \label{def:xrangefunctional}
\psi(b):= \max_{t\le \b} \big|b_x(t) - b_x(0)\big|.
\ee
\tcb{This quantity  is one of the main players in the proof of our main result. In particular, we stress that the bound $-\log (\mathbb P(|\psi(b)| > r )) \approx cr^2$ will be the key to handling fast Brownian bridges.}
The following two lemmas provide bounds on the configurations which are the worst offenders in crossing boundaries.}
\begin{lemma}[$\Th$ contains most configurations]
	\label{fubRangeLem}
	Let \tcb{$k \ge 1$} and $I_1, \dots, I_k \su \R$ be distinct intervals with integer endpoints and of length $R \ge 1$. 
	For every $\CC \in \Conf(\La_N)$,
	$$\#\{i \le k:\, \CC \not \in \Th({I_i} \ti D)\} \le 2\sum_{b \in \CC} \big(\psi(b) - R^{7/8} + 2\big)_+.$$
\end{lemma}

\begin{lemma}[Bound on Brownian bridge $x$-ranges] \label{lem:entropy-estim}
	Let $P\in \mc P_s$ with $\ent(P) < \ff$. 
	{There exists $c > 0$, not depending on $P$, such that for all $\z > 1$ 
	\be
	\sup_{\substack{\La \su S \\ |\La| > 0}}	\frac1{|\La|}\E_{P_\La}\Big[ \sum_{b \in \CC}(\psi(b)^{7/6} - \z)_+\Big] 
	\le (c+\ent(P))\z^{-5/8}.
\ee}
\end{lemma} 

\tcb{Our choice of the exponent 7/6 on the left side of the above inequality is due to later use of this lemma (see Definition \ref{def:regularity}).}

{As usual the supremum on the left-hand side is over domains of the form $[L_-,L_+]\times D$ with $L_\pm \in \Z$.} {The lemma is similar to Lemma~5.2 in \cite{georgii2}; it is based on entropy bounds. Some context for this kind of bound is provided in Section~\ref{sec:entropy}.}

Regarding the exponent choices in the above lemma, see Remark \ref{rem:exponent choice}. Before proving Lemmas \ref{fubRangeLem} and \ref{lem:entropy-estim}, let us see how they yield Proposition \ref{superAddLem}.

\begin{proof}[Proof of Proposition \ref{superAddLem}]
 Put 
 \begin{align}\label{def:WMN}
 W^M_N(\CC) := \big(N^{-1}\tH_{\La_N}(\CC)\big) \wedge M.
 \end{align}
	Let $\{N_i^-\}_{i \ge 1}$ and $\{N_i^+\}_{i \ge 1}$ be sequences realizing the $\liminf$ and $\limsup$ {under $\E_P$}. That is,
	\begin{align*}
		\lim_{i \to \ff} \E_P\big[W^M_{N_i^-}\big] & = \liminf_{N \to \ff}\E_P\big[W^M_N\big]\\
		\lim_{i \to \ff} \E_P\big[W^M_{N_i^+}\big] & = \limsup_{N \to \ff}\E_P\big[W^M_N\big]
	\end{align*}
	where we may assume without loss of generality that $N_i^- \ge (N_i^+)^2$ for all $i \ge 1$. Now, we set $N_i' = \lceil (N_i^- - 4({N_i^-})^{7/8}) \rceil$ and claim the almost-superadditive bound
	\begin{align}
		\label{supAddEq}
		N_i'\E_P\big[W^M_{N_i^+}(\CC)\big] \le N_i^-\E_P\big[W^M_{N_i^-}(\CC) \big] + 2M\E_P\Big[\sum_{b \in \CC_{\La_{N_i^-}}}\big(\psi(b) - (N_i^+)^{7/8} + 2\big)_+\Big].
	\end{align}
	\tcb{	First, note that for any $r > 0$ and for any $R > 0$ with $R^{7/8}-2 \geq (R/2)^{7/8}$ we have }
$$
	(r - R^{7/8} + 2)_+ \le r \one\{r \ge (R/2)^{7/8}\}\le r^{7/6} \one\{r^{7/6} \ge (R/2)^{49/48}\}\le 2(r^{7/6} - (R/2)^{49/48}/2)_+.
$$
	{Now, divide both sides of \eqref{supAddEq} 	by $N_i^-$ and send $i \to \ff$.} Lemma \ref{lem:entropy-estim} shows that
	\begin{align}\label{eq:exponent choice}
		\frac1{N_i^-}\E_P\Big[\sum_{b \in \CC_{\La_{N_i^-}}}\big(\psi(b) - (N_i^+)^{7/8} + 2\big)_+\Big]
		&\le \frac2{N_i^-}\E_P\Big[\sum_{b \in \CC_{\La_{N_i^-}}}\big(\psi(b)^{7/6} - (N_i^+/2)^{49/48}\big)_+\Big]\nn\\
		&\le \tcb{2(c + \ent(P))(N_i^+/2)^{-(5\cdot 49)/(8\cdot 48)}}\to 0
		\end{align}
	which concludes the proof, modulo the claim \eqref{supAddEq}.

	To establish \eqref{supAddEq}, we {let $K_1, \dots, K_{N_i'} \su [0, N_i^-]$ be $N_i'$  intervals of length $N_i^+$, placed within the interval $[0, N_i^-]$, such that $K_{i + 1}$ is obtained by shifting $K_i$ to the right by 1 (along the $x$-coordinate) and such that each $K_j$ is at least distance $\lceil (N_i^-)^{7/8}\rceil$ from $0$ and $N_i^-$.
Now, let $\CC \in \Conf(K_j)$  be arbitrary, and $\La \supset K_j$, $\CC'\in \Neut(\La)$ be such that $\CC'\in \Th(K_j)$ and $\CC'_{K_j} = \CC_{K_j}$.} 
	Then by definition, see \eqref{eq:wkdef},
	\begin{align}
	\tH_{K_j}(\CC') \wedge N_i^+M \le \frac1{2\b} \int_{K_j \ti [0, \b]} \bigl|\nabla V_t(z, \CC', \La)|^2 \d(z, t) + MN_i^+\one\{\CC \not \in \Th(K_j)\}.
	\end{align}
	Summing over all $j \le N_i'$ and dividing by $N_i^+$ yields that
	$$	\sum_{j \le N_i'}\big(\frac1{N_i^+}\tH_{K_j}(\CC')\big) \wedge M \le \frac1{2\b} \int_{[0, N_i^-] \ti [0, \b]} \bigl|\nabla V_t(z, \CC', \La)|^2 \d(z, t) + M\sum_{j \le N_i'}\one\{\CC \not \in \Th(K_j)\}.$$
	applying Lemma \ref{fubRangeLem} and taking expectations concludes the proof of \eqref{supAddEq}. 
\end{proof}
Next, we show that the free energy is lower semicontinuous.

\begin{proof}[Proof of Proposition \ref{rateFunProp}]
 {The specific relative entropy is lower semicontinuous by arguments similar to \cite[Proposition 2.6]{georgii2}. Hence,  if $\ent(P) = \infty$, then for any level $\ell$ there is a neighborhood $U$ of $P$ such that $P'\in U$ implies $\ent(P')>\ell$ thus we have lower semicontinuity of $\mc F$ at such $P$. Therefore, we may assume that $\ent(P) < \infty$ from now on.}

Since the supremum of lower semicontinuous functions remains lower semicontinuous, it suffices to show lower semicontinuity of the truncated free energy
$$\mc F^M = \W^M + \b^{-1}\ent.$$
Let $P \in \Ps$ and $\e > 0$ be arbitrary. We need to show that there exists a neighborhood $U$ of $P$ such that
\begin{align}
\label{lsc_eq}
\mc F^M(P) \le \inf_{P' \in U}\mc F^M(P') + \e.
\end{align}
 
	Since the left-hand side of \eqref{lsc_eq} is at most \tcb{$\a:= M + \ent(P)/\beta$} (with $\ent(P)<\infty$), it suffices to show that
$$\mc F^M(P) \le \inf_{\substack{P' \in U\\ \ent(P') \le \a}}\mc F^M(P') + \e.$$
By lower semicontinuity of the specific relative entropy, we may focus on the energy part. Here, for every $M, N \ge 1$ the observable $W^M_N(\CC)$ defined in \eqref{def:WMN} is local and bounded, so that 
$$
	P \mapsto f_N(P):= \E_P[W^M_N(\CC)]
$$
is continuous in $P$. Moreover, the proof of Proposition \ref{superAddLem} revealed that the energies are almost super-additive. If they were precisely super-additive, then the limit in $N$ could be replaced by a supremum, so that again we could leverage lower semicontinuity of a supremum of continuous functions. In the rest of the proof, we show that this line of argumentation extends to the present setting of almost super-additivity.

More precisely, we provide a neighborhood $U$ of $P$ such that
$$f(P) \le \inf_{\substack{P' \in U\\ \ent(P') \le \a}}f(P') + \e,$$
where we set 
$$
	f := \lim_{N \to \ff} f_N=\mathbb W^M.
$$
To achieve this goal, we decompose $f(P) - f(P')$ as
\begin{align*}
	f(P) - f(P') = \big(f(P) - f_N(P)\big) + \big( f_N(P) - f_N(P')\big) + \big(f_N(P') - f(P')\big).
\end{align*}
The first term on the right-hand side tends to 0 as $N \to \ff$. For fixed $N$, the second term becomes arbitrarily small for all $P'$ from a sufficiently small neighborhood of $P$. Hence, it suffices to show that for fixed $a > 0$
$$\limsup_{N \to \ff} \sup_{N' \ge N^2}\sup_{P':\, \ent(P') \le a}(f_N(P') - f_{N'}(P'))\le 0.$$
Now, copying the proof of the bound \eqref{supAddEq} gives that
$$(1 - 4(N')^{-1/8})f_N(P') \le f_{N'}(P') + M\E_{P'_{\La_{N'}}}\Big[\sum_{b \in \CC' }\big(\psi(b) - N^{7/8} + 2\big)_+\Big].$$
Hence, applying Lemma \ref{lem:entropy-estim} concludes the proof.
\end{proof}

It remains prove Lemmas \ref{fubRangeLem} and \ref{lem:entropy-estim}.
\begin{proof}[Proof of Lemma \ref{fubRangeLem}]

The $\{I_j\}_{j \le k}$ are each of the form $\th_i(\La_R)$ for some $i \in \Z$.
	We have
	\begin{align*}
		\sum_{i \in \Z} \one\{\CC \not \in \Th(\th_i(\La_R))\} &\le \sum_{i \in \Z}\sum_{b \in \CC} \one\{b(0) \not \in \th_i(\La_R)\text{ and } b \cap \th_i(\La_R^-) \ne \es\}\\
		&\le \sum_{b \in \CC}\#\{i \le - b_x(0):\, b \cap \th_i(\La_R^-) \ne \es\}\\
		&\phantom \le + \sum_{b \in \CC}\#\{i \ge - b_x(0) - R:\, b \cap \th_i(\La_R^-) \ne \es\}.
	\end{align*}
	 Now, if $i \le -b_x(0)$, then $b$ can intersect $\La_R^-$ only if $i \ge -b_x(0) - \psi(b) + \lceil R^{7/8}\rceil$. Hence, the total number of such $i$ is at most $(\psi(b) - R^{7/8} + 2)_+$. A similar argument applies in the case $i \ge - b_x(0) - R$ which gives the first bound in the lemma. 
\end{proof}

\begin{proof} [Proof of Lemma \ref{lem:entropy-estim}]
	We adapt an argument from the proof of \cite[Lemma 5.2]{georgii2}.	By \cite[Eq.~(2.11)]{georgii2},
	the {\it relative entropy} $\ent_\La(P_\La)$ of $P_\La $ with respect to $\Pois_\La$ satisfies 

	\be
	\label{varCharEq}
		{\ent_\La}(P_\La) = \sup_f \big\{\E_{P_\La}[f(\CC)] - \log \E_{\Pois_\La}[\exp(f(\CC))]\big\}
	\ee
	where the supremum is over all bounded measurable $f:\,\Conf(\La)\to \R$. The inequality extends to non-negative measurable $f$ that are not necessarily bounded \cite[Eq.~(5.3)]{georgii2}.	Set $a = \z^{5/8}$, $$f(\CC) = \sum_{b\in \CC } (\psi(b)^{7/6} - \z)_+ \, ,$$
	and 	apply \eqref{varCharEq} to $af$. {Noting that the specific relative entropy $\ent$ defined in (\ref{def:entP}) is equal to $\sup_\La |\La|^{-1} \ent_\La$ (see Remark 2.5 in \cite{georgii2}),} we find that
	\be \label{eq:vacha1}
		\frac1{|\La|}\, \E_{P_\La}[f(\CC_\La)] 
			\le \frac1a \ent(P) + \frac1{a|\La|}\, \log \E_{\Pois_\La}\big[ \exp\bigl(af(\CC)\bigr) \big]. 
	\ee
	Then, inserting the Laplace transform of a marked Poisson point process \cite[Prop. 5.4]{last2017lectures},
	$$		\log\E_{\Pois_\La}\big[ \exp(af(\CC))\big] = |\La|\bigl( \exp({\vp(\z)}) - 1\bigr) ,		 $$
			 where for a Brownian bridge $B$ on $[0, \b]$, 
	\be
		\vp(\z) := \log {\E_\BB}\Bigl[\exp\bigl( a (\psi(B)^{7/6} - \z)_+\bigr)\Bigr].
	\ee
	We next evaluate for $\z>0$,
	\begin{align} \label{eq:regular-expmom}
		e^{\vp(\z)} - 1 &\nn \le {\E_\BB}\bigl[\exp({a \psi(B)^{7/6}}) \one{\{\psi(B)>\z^{6/7}\}}\bigr]\\
			&\le \Bigl( {\E_\BB}\bigl[\exp({2a \psi(B)^{7/6}})\bigr]\, \P(\psi(B) >\z^{6/7})\Bigr)^{1/2}. 
	\end{align}
	Tail estimates on the maximum of a Brownian motion (\cite[Thm 1.2.6]{durrett2019probability},\cite[Section 2.2.1]{bmBook}), show that $$\log \P(|\psi(B)| > \z^{6/7})$$ is of order $-\z^{12/7}$ for large $\z$. {Next, applying the Laplace principle to the function $f(x)=2ax^{7/6}-cx^2$, which is maximized when $a^{12/5}\sim x^2$, shows that the exponential moment ${\E_\BB}\bigl[e^{2a \psi(B)^{7/6}}\bigr]$ is of order $\exp(O(a^{12/5}))$ as $\z \to \ff$. In particular, $e^{\vp(\z)} - 1$ remains bounded,} which concludes the proof.
\end{proof}

%%%%%%%%%%%%%%%%%%%%%%%%%
%%%%%%%%%%%%%%%%%%%%%%%%%
%%%LDP UPPER BOUND
%%%%%%%%%%%%%%%%%%%%%%%%%
%%%%%%%%%%%%%%%%%%%%%%%%%
\subsection{Proof of Theorem \ref{mainThm} -- LDP upper bound}
\label{upSec}
 Using \eqref{splitEq}, we re-write $\mun(\Emp(\CC) \in A)$ for $A \su \Pu$ measurable as 
$$\mun(\Emp(\CC) \in A) = \frac 1{Z_{N, \b}}e^{N\b g^{\eta, \eta}(o)} \E_N\big[\one\{\Emp(\vgN) \in A\} \exp\big(- \b H_{\La_N}(\CC) \big) \big].$$
After dividing by $\mun(\Emp(\CC) \in \Pu) = 1$, we arrive at
\begin{align}
	\label{fieldProbEq}
	\mun(\Emp(\CC) \in A) = \frac{\E_N\big[\one\{\Emp(\vgN) \in A\}
	 \exp\big(- \b H_{\La_N}(\CC) \big)	\big]}{\E_N\big[ \exp\big(- \b H_{\La_N}(\CC) \big)\big]}, 
\end{align}
so that proving Theorem \ref{mainThm} reduces to understanding the asymptotics of the numerator in \eqref{fieldProbEq} at an exponential scale.

To prove the LDP upper bound, we leverage a lower bound for the energy, large deviations for Poisson point processes, and standard arguments from the proof of the upper bound in Varadhan's integral lemma \cite{dz98}. As a critical ingredient for the proof of the upper bound we use that if the particles do not interact, then the rate function is given by the relative entropy.
\begin{lemma}[LDP upper bound without interaction]
	\label{biLDPLem}
	For every measurable $A\su \mc P$, 
	$$\limsup_{N \to \ff} \frac1N \log\P_N(\Emp(\CC) \in A) \le -\inf_{P\in \ol A} \ent(P).$$
\end{lemma}

Recall from (\ref{def:N(P)}) that $\mc U_{\text{meas}}(P)$ is the set of measurable neighborhoods around $P$.

\begin{lemma}[{Measurable} neighborhoods of high-enough energy] \label{lem:lscersatz}
{ 	Let $M \ge 1$ and $\e > 0$ be given. For each $P\in \Ps$, there exists a measurable neighborhood $U_{F, P}\in \mc U_{{\mathrm{meas}}}(P)$ such that one can choose $R(M, \e, P)$ and then $ N(M, \e, P, R)$ to satisfy the following property. If $\CC \in \Neut(\La_N)$ is such that $\Emp(\CC) \in U_{F, P}$, then for all $R\ge R(M, \e, P)$ and $N\ge N(M, \e, P, R)$,
	\be \label{eq:lowernei}
	\frac1N H_{\La_N}(\CC) \ge \W^M(P) - \e- \frac{2M}N\sum_{b \in \CC}\big(\psi(b) - R^{7/8} + 2\big)_+.
	\ee}
\end{lemma} 

{Lemma~\ref{lem:lscersatz} and Proposition~\ref{enBoundProp} below allow us to apply arguments similar to the proof of Varadhan's lemma, based on a large deviation principle for the empirical field, even though the energy is not directly a function of the empirical field. In~\cite{georgii2}, Georgii and Zessin deal with this difficulty by introducing a notion of \emph{asymptotic empirical functional} \cite[Eq.~(3.1)]{georgii2}, however the latter asks for some uniformity condition which cannot be applied to Coulomb systems.}

Before proving Lemmas \ref{biLDPLem} and \ref{lem:lscersatz}, we explain how they enter the proof of Theorem \ref{mainThm}. By establishing lower and upper bounds for the {non-normalized} measures in \eqref{fieldProbEq}, it is enough to control the numerator in \eqref{fieldProbEq}. 

\begin{proof}[Proof of Theorem \ref{mainThm} -- upper bound {for the numerator in \eqref{fieldProbEq}}]
Let $A\su \mc P$ be measurable. Because the rate function is lower semicontinuous and $\W^M$ is increasing in $M$, it is enough to show that for each $M\ge1$,
\be\label{eq:numerator ineq}
\limsup_{N\to \ff} \frac1N \log \E_N\big[\exp(-\b H_{\La_N} (\CC)) \one\{\Emp(\CC) \in A\}\big] \le - \inf_{P\in \ol A} (\b \W^M(P) + \ent(P)).
\ee
Indeed, if there is a finite constant $C$ such that for every $M \ge 1$ there exists $P^M \in \bar A$ with
$$\W^M(P^M) + \b^{-1}\ent(P^M) \le C,$$
then also
$$\W(P^*) + \b^{-1}\ent(P^*) = \lim_{M \to \ff}
\W^M(P^*) + \b^{-1}\ent(P^*) \le C,$$
where $P^* \in \bar A$ is any limit point of the $P^M$. Here, the compactness of the entropy-sublevel set
$$K_\a := \{P\in \mc P:\,\ent(P)\le \a\}$$
guarantees that some limit point $P^*$ exists \cite[Proposition 2.6]{georgii2}.

       By compactness, there exists a finite set $\{P_i\}_{i \le k}$ and measurable neighborhoods $U_i = U_{F_i,P_i}$ which cover ${K_\a\cap \ol A}$ and satisfy the inequality from Lemma \ref{lem:lscersatz}, for  $N \ge \max_i N(P_i,\e)$.

For $\e > 0$, set
$$E_{R,\e} := \big\{P \in \Pu:\, \E_P[\sum_{b:b_x(0)\in[0,1]}(\psi(b) - R^{7/8} + 2)_+] \leq \e\big\}.$$
 On the event that $\{\Emp(\CC) \in E_{R,\e}\}$, the error in Lemma $\ref{lem:lscersatz}$ is small.

{The set $E_{R,\varepsilon}$ is measurable. Indeed, let $f:\mathsf{Conf}\to \R_+$ be a measurable local function and $A:= \{P\in \Pu: \E_P[f]\leq \e\}$.  Then for every $n\in \N$, the set $A_n:= \{P\in \Pu: \E_P[f\wedge n] \leq\e\}$ is the preimage of $[0,\e]$ under the map $P\mapsto \E_P[f\wedge n]$, which is measurable because $f\wedge n$ is local and bounded. Thus $A_n$ is measurable. Moreover $A_n\downarrow A$, hence $A$ is measurable as well. The obvious choice of $f$ yields that $E_{R,\e}$ is indeed measurable.
}

	\tcb{First, we recall from the proof of Proposition \ref{superAddLem} that for large enough $R > 0$ and all $r > 0$ we have $(r - R^{7/8} + 2)_+ \le 2(r^{7/6} - (R/2)^{49/48}/2)$.}
Hence, we conclude from {the reasoning below Eq.~\eqref{eq:vacha1} in the proof of Lemma~\ref{lem:entropy-estim}} that 
$$\frac1N\log\E_N\Big[\exp\Big(\tfrac12(\tfrac R2)^{\tfrac58 \cdot \tfrac{49}{48}}\E_{\Emp(\CC)}\Big[\sum_{b:b_x(0)\in[0,1]}(\psi(b) - R^{7/8} + 2)_+\Big]\Big)\Big]$$
remains bounded as $R \to \ff$. Thus, by the Markov inequality,
$$\lim_{R \to \ff}\limsup_{N\to\ff}N^{-1}\log\P_N(\Emp(\CC) \in E_{R,\e}^c) = - \ff.$$
	In particular, \tcb{we may} restrict to the set $A \cap E_{R,\e}$ for $R$ large enough. 

 The remaining part adapts the arguments of \cite[Lemma 4.3.6]{dz98}. By compactness of $\ol A\cap K_\a$, there exists a finite covering $\cup_{i \le k} U_i \supset \ol A \cap K_\a$ by measurable neighborhoods $U_i = U_{F_i,P_i}$ satisfying the inequality from Lemma \ref{lem:lscersatz} for $R_i$ chosen appropriately and $N \ge \max_i N(M,\e, P_i,R_i)$.
We now bound the numerator
in \eqref{fieldProbEq}.
For all $N \ge \max_i N(R_i, M, P_i, \e)$, we have 
\begin{equation}
\begin{aligned}
\label{up_bound_eq}
&\E_N\big[\exp(-\b H_{\La_N} (\CC)) \one\{\Emp(\CC) \in A \cap E_{R, \e}\}\big] \\
&\quad\le \sum_{i \le k} \mathrm e^{- \b N \W^M(P_i)- 3\e M} \P_N\big(\Emp(\CC) \in U_i\big) + \P_N\Big(\Emp(\CC)  \in \bigl(\bigcup_i U_i\bigr)^c \cap A \Big).
\end{aligned}
\end{equation}
	Since $(\bigcup_i U_i)^c\cap \bar A \su K_\a^c$,  \tcb{we deduce from Lemma \ref{biLDPLem} that}
              $$\limsup_{N\to \ff} \frac1N \log \P_N\Bigl(\Emp(\CC)  \in \bigl(\bigcup_i U_i\bigr)^c \cap A \Bigr) \leq - \inf_{P\in K_\a} \ent(P) < -\a,$$
	which becomes arbitrarily negative for large $\a$. 
In light of Lemma \ref{biLDPLem}, taking the limit $N\to \ff$ first and then $\e\to 0$ in \eqref{up_bound_eq} proves \eqref{eq:numerator ineq}.
\end{proof}

Next, we prove Lemma \ref{biLDPLem}.
\begin{proof}[Proof of Lemma \ref{biLDPLem}]
		Let $\Pois$ be the reference measure introduced above \eqref{def:entP}. Then 
	\be
		\P_N(\Emp(\CC) \in A) = \Pois_{\La_N}\big( \Emp( \CC) \in A \, |\, \#\CC = N\big) 
	\ee
	hence
	\be
		\frac1N \log \P_N(\Emp(\CC) \in A) \le \frac1N \log \Pois_{\La_N}( \Emp( \CC) \in A) - \frac1N \log \Big( \frac{N^N}{N!} \exp(- N) \Big).
	\ee
	The last term vanishes by Stirling's formula. For the probability under the Poisson process, putting $Z_i := \CC_{[i, i + 1) \ti \D}$ and thinking of $\Emp(\CC)$ as the empirical field of the independent spin system $\{Z_i\}_{i \in \Z}$ on $\Z$, we conclude with \cite[Theorem 1.2]{georgii1}.
\end{proof}

Finally, to prove Lemma \ref{lem:lscersatz}, we need one further auxiliary result. {Set
\begin{align}\label{def:WRM}
	\WRM(P) := \E_P\big[ \big(R^{-1}\tH_{\La_R}(\CC)\big) \wedge M\big] 
\end{align}
to be the expected truncated specific energies in $\La_R$. Note that by {Definition~\ref{specEnDef}}, $\lim_{R\to\ff}\WRM(P)=\WMM(P)$.}
\begin{lemma}[{Lower bound for the energy}]
	\label{semContLem}
	Let $1 \le R \le N - 1$ and $\CC \in \Neut(\La_N)$ be a charge-neutral bridge configuration. Then,
	$$H_{\La_N}(\CC) \ge (N - R)\WRM(\ms{Emp}_{N - R}(\CC)) - 2M\sum_{b \in \CC}\big(\psi(b) - R^{7/8} + 2\big)_+.$$
\end{lemma}
\begin{proof} 
	The definition of $\ms{Emp}$ gives that
	$$(N - R)\WRM(\ms{Emp}_{N - R}(\CC)) = \sum_{ i \le N - R - 1} \big(R^{-1}\tH_{\La_R}(\th_i \CC)\big) \wedge M.$$ 
	Since $\La_R \su \th_i(\La_N)$ for every $ i \le N - R - 1$, {the definition of $\tH_{\La_R}$ and the fact that $\CC\in\Neut(\La_N)$} yields that
	\begin{align}
	\tH_{\La_R}(\th_i \CC)\wedge RM \le \frac 1{2\b} \int_{\th_{-i}(\La_R) \ti [0, \b]}|\nabla V_t(z, \CC, \La_N)|^2 \d(z, t) + \one\{\CC \not \in \Th(\th_{-i}(\La_R))\} RM .
	\end{align}
	In particular, dividing by $R$ and summing over $ i \le N - R - 1$ yields that
	\begin{align*}
		(N - R) \WRM (\ms{Emp}_{N - R}(\CC)) &\le \frac1{2R \b}\sum_{ i \le N - R - 1}\int_{\th_{-i}(\La_R) \ti [0, \b]}|\nabla V_t(z, \CC, \La_N)|^2 \d(z, t)\\
		&\phantom\le + M\sum_{ i \le N- R - 1} \one\{\CC \not \in \Th(\th_{-i}(\La_R))\},
	\end{align*}
	where the first sum on the right-hand side is bounded by 
	\begin{align*}
	\frac 1{2\b}\int_{\T \ti [0, \b]}|\nabla V_t(z, \CC, \La_N)|^2 \d(z, t) = H_{\La_N}(\CC).
	\end{align*}
Applying Lemma 	\ref{fubRangeLem} concludes the proof.
\end{proof}

\noindent

\begin{proof} [Proof of Lemma \ref{lem:lscersatz}]

	There exists 	$R = R(M, \e, P) > 0$ such that 
	$$
		\sup_{R \ge R(M, \e, P)}\big|\WRM(P) - \W^M(P)\big| < \e/3.
	$$
	 Note that {$\tH_{\La_R}\wedge M$} is a bounded local observable. Thus, $\WRM(Q)$ are arbitrarily close to $\WRM(P)$ for all $Q$ belonging to a sufficiently small open neighborhood $U_{F, P}$ of the form \eqref{topology-basis}.
 
	 Also note that $\WRM(\ms{Emp}_{N - R}(\CC))$ gets arbitrarily close to $\WRM(\Emp(\CC))$ if $N$ is sufficiently large. Hence, 
	 $$\big|\WRM(\ms{Emp}_{N - R}(\CC)) - \W^M(P)\big| < 2\e/3$$ 
	 if $\Emp(\CC) \in U_{F, P}$. In particular, by Lemma \ref{semContLem},
	 	 \be
		H_{\La_N}(\CC) \ge (N- R)\big(\W^M(P) - 2\e\big/3) - 2M\sum_{b \in \CC}\big(\psi(b) - R^{7/8} + 2\big)_+
	 \ee
	for all $\CC\in U_{F, P}$. Dividing both sides by $N$ and choosing $N$ large enough so that $\frac RN \W^M(P)<\e/3$
	proves \eqref{eq:lowernei}.
\end{proof} 

%%%%%%%%%%%%%%%%%%%%%%%%%
%%%%%%%%%%%%%%%%%%%%%%%%%
%%%LDP LOWER BOUND
%%%%%%%%%%%%%%%%%%%%%%%%%
%%%%%%%%%%%%%%%%%%%%%%%%%
\subsection{Proof of Theorem \ref{mainThm} -- LDP lower bound}
\label{lowSec}
In order to prove the lower bound in Theorem \ref{mainThm}, we establish a quasi-continuity\footnote{This terminology is {taken} from \cite[Section 4.1.2]{serfInv} {and \cite{bodineau-guionnet1999}}.} property of the specific energy. To make this precise, we show the following analog of \cite[Proposition 4.2]{serfInv}. In essence, Proposition \ref{enBoundProp} states that it is possible to construct a family of bridge configurations, with substantial probability mass, such that (i) the associated empirical fields are close to a given $P \in \Pst$, and (ii) the energy of the bridge configurations does not exceed $\W(P)$ substantially.

\begin{proposition}[Quasi-continuity of the specific energy] 		
	\label{enBoundProp}
		Let $P \in \Pst$ be such that $\mc F_\b(P) < \ff$. Then, for every measurable neighborhood $U_P\in \mc U_{\text{meas}}(P)$ and sufficiently small $\de > 0$, 
		\be
			\liminf_{N\to \ff}\frac1N \log \P_N\Bigl( \Emp(\CC) \in U_P, \ N^{-1} H_{\La_N}(\CC) \le \W(P) + \de \Bigr) \ge - \ent(P) - {O(\de)}. 
		\ee
\end{proposition}

	The proof of Proposition \ref{enBoundProp} is the heart of the present paper. Before establishing Proposition~\ref{enBoundProp}, we show how it completes the proof of Theorem \ref{mainThm}.

\begin{proof}[Proof of Theorem \ref{mainThm}, lower bound {for the numerator in~\eqref{fieldProbEq}}]
	Let $A\su \mc P$ be a measurable set, $P\in A^\circ \cap \Ps$ with \mbox{$\mc F_\b(P)<\ff$,} and $\de>0$. Then, there exists a measurable neighborhood $U_P$ of $P$ such that $U_P\su A$, see (\ref{topology-basis}). Note
 {\begin{align*}
	 &\frac1N \log \E_N\big[\one\{\Emp(\vgN) \in A\} \exp\big(- \b H_{\La_N}(\vgN) \big)\big]\\ 
	 &\quad\ge \frac1N \log \E_N\big[\one\{\Emp(\vgN) \in A, \ N^{-1} H_{\La_N}(\CC) \le \W(P) + \de \} \exp\big(- \b H_{\La_N}(\vgN) \big)\big]\\ 
	 &\quad \ge -{\b O(\de)} - \b \W(P) + \frac 1N \log \P_N( \Emp(\vgN) \in U_P, \ N^{-1} H_{\La_N}(\CC) \le \W(P) + \de ).
 \end{align*}}
By Proposition \ref{enBoundProp}, 
 \be
 	\liminf_{N\to \ff} 	 \frac1N \log \E_N\big[\one\{\Emp(\vgN) \in A\} \exp\big(- \b H_{\La_N}(\vgN) \big)\big] \ge - (\b \W(P) + \ent(P)) - O( \de).
 \ee
	This holds true for every $\de>0$ and $P\in A^\circ \cap \Ps$, hence 
	\be
		 	\liminf_{N\to \ff} 	 \frac1N \log \E_N\big[\one\{\Emp(\vgN) \in A\} \exp\big(- \b H_{\La_N}(\vgN) \big)\big] \ge - \inf_{P\in A^\circ \cap \Ps} \bigl( \b \W(P) + \ent(P)\bigr).
	\ee
 	On $A^\circ\setminus \Ps$ the free energy $\mc F_\b$ is infinite, therefore the infimum over $A^\circ\cap \Ps$ is equal to the infimum over $A^\circ$ and the proof is complete.
\end{proof}

\section{Entropy and energy estimates} \label{sec:energy entropy}

In order to implement the screening construction, we require a variety of subtle entropy and energy bounds. The present section is devoted to this topic. First, recall that a key ingredient in the theory of Gibbs measures is a control on interactions between distinct regions of space. In our setup, if we split the strip $S$ into two halves separated by an interface at $x = x_0$, we would like to control interactions between the left and right half-strips $x<x_0$ and $x>x_0$. Interactions can be large in particular if there are too many bridges crossing the interface, or too many bridges that have their initial point $b(0)$ close to $x_0$. This section gathers the required estimates and explains how the estimates connect to classical problems in statistical mechanics.

{The first type of bound is an entropy bound and it has actually already been stated and proven in Lemma~\ref{lem:entropy-estim}. Section~\ref{sec:entropy} provides some additional physical context.}

Some events come with large energy penalties and can be excluded if we impose bounds on the energy. This {second type of bound} is exploited in Section \ref{sec:energybounds}. 

\subsection{{Entropy bounds:} Brownian bridges as unbounded spins} \label{sec:entropy} 

To get a feel for the difficulties caused by Brownian bridges, we may split the energy into a contribution from time-constant bridges {(which amounts to classical mechanics)} plus {quantum} correction terms {of Boltzmann-type} coming from the bridges not being constant, 
\be
	H_\La(\CC) = \frac12 \int_\La |\nabla V_0(z, \CC, \La)|^2 \d z + \frac1{2\b}\int_{\La\ti [0, \b]} \bigl(|\nabla V_t(z, \CC, \La)|^2 - |\nabla V_0(z, \CC, \La)|^2\bigr) \d (z, t). 
\ee
The first term depends only on the initial positions of the bridges, the second term depends on the actual bridges. It is instructive to look at the bridge-dependent part of a pair interaction: let $b, \gamma$ be two bridges, then 
\be \label{eq:bridges-pairint}
	\int_0^\b g\bigl( b(t) - \gamma(t)\bigr) \d t 
		 = \b g\bigl( b(0) - \gamma(0)\bigr) + 	\int_0^\b \Bigl( g\bigl( b(t) - \gamma(t)\bigr)- g\bigl( b(0) - \gamma(0)\bigr)\Bigr) \d t. 
\ee
Consider for simplicity the purely one-dimensional system with Green's function $g(x) = - |x|/2$. Then the additional contribution to the pair interaction from bridge degrees of freedom, i.e., the second contribution in \eqref{eq:bridges-pairint}, is equal to 
\be
	- \frac12 \int_0^\b \Bigl( |b(t) - \gamma(t)| - |b(0) - \gamma(0)|\Bigr) \d t, 
\ee
which can in principle be arbitrarily large for paths with large values for $b(t) - b(0)$ or $\gamma(t)- \gamma(0)$, regardless of the distance $|b(0) - \gamma(0)|$. However, this should be unlikely because Brownian bridges in the bounded time interval {$[0, \b]$} should typically not stray away too much from their initial point (this unlikeliness is at the level of entropy---without even having to take energy into account). Hence, a good control for the additional interaction induced by the bridge degrees of freedom requires probabilistic entropy estimates. 

This situation is similar to the classical statistical mechanics for systems of unbounded spins on a lattice. The space $C([0, \b])$ is a non-compact single-spin space, the path $(b(t)- b(0))_{t\in [0, \b]}$ with initial point $0$ plays the role of a spin. Ensuring the existence of Gibbs measures for unbounded spin systems requires additional conditions and estimates. Several approaches are possible. \emph{Superstability estimates} build on a long series of technical estimates that culminate in bounds on the probability of large spin values with respect to the Gibbs measure \cite{ruelle70, ruelle76, lebowitz-presutti76}. For interacting diffusions seen as Gibbs measures on path space, a different condition is used~\cite{fritz87gradient, dereudre2003interacting-brownian}. In our setting, we employ the the set $\Theta(K)$ in Eq.~\eqref{eq:thetadef} together with the \emph{entropy bound} in Lemma \ref{lem:entropy-estim} for stationary measures, {which is a variant of Lemma~5.2 in \cite{georgii2}}. In particular, this bound plays a role in Lemma \ref{expAppCor} where we show that, with high probability, configurations behave in a reasonable way.
We will additionally introduce notions of $(M,\e)$-regular configurations and tame abscissas (Definitions~\ref{def:regularity} and~\ref{def:etame}).

\subsection{Energy bounds via charge imbalance} \label{sec:energybounds}

Proceeding in the vein of \cite{aizenman2010symmetry}, we now rely on energy estimates to exclude configurations that have too many bridges close to a given abscissa $x = x_0$. Intuitively, if there are too many bridges nearby $x = x_0$, then there is large charge imbalance, hence a large electric field and large electrostatic energy. Thus, conversely an energy bound constrains the charge accumulating in any finite region. 

As we put periodic boundary conditions in the $y$-direction, we henceforth write $\pa (I \ti D) = (\pa I) \ti D$ for the $x$-boundary of a domain inside $S$.
Henceforth, set $\vec e_x = (1, 0, \dots, 0)\in \R^{k+1}$ and
\be \label{def: mu_t}
\mu_t(\CC, \d z) := - \sum_{b\in \CC} \de_{b(t)}^\eta(\d z) + \one_\La\d z,
\ee 
so that $ \mu_t(\CC, A) = \mu_t(\CC)(A) $ computes the {\it charge imbalance} {or \emph{net charge}} in a region $A \su S$. {(We use here the terms ``charge imbalance'' and ``net charge'' synonymously but note that in previous work \cite{aizenman2010symmetry} charge imbalance instead refers to the difference of the total charge to left of some abscissa minus total charge to the right.)}

\tcb{Lemma \ref{lem:chargimb1} below is closely related to the discrepancy bounds \cite[Lemma 2.2]{petrache}. The essence of both results is that the electric energy in a domain grows at least in the order of the square of the charge imbalance. Whereas \cite[Lemma 2.2]{petrache} contains refined estimates that also apply to Riesz energies, the focus of Lemma \ref{lem:chargimb1} concerns time-varying bridges.
}
\begin{lemma}[Charge imbalance bound]
	    \label{lem:chargimb1}
	Let $ L_- \le x_- \le x_+  \le L_+$ and $\CC \in \Neut(\La)$ with $\La = [L_-, L_+]{\times D}$. Then,
	\been
	\im  For $\La_- := [L_-, x_-] \ti D$,
			    \begin{align*}
				    \big|\mu_0(\CC, \La_-)\big|        \le \Bigl( \frac1\b\int_{\D\ti [0, \b]} |\nabla V_t((x_-, y), \CC, \La)|^2 \d (y, t) \Bigr)^{1/2} 
							            + \#\big\{b\in \CC:\, \inf_{t \le \b} |b_x(t) - x_-| \le 1\big\}.
										    \end{align*}
	\im For $\La' :=[x_-,  x_+] \ti D $,
	\begin{multline*}
	\big	|\mu_0(\CC, \La')\big|		\le \Bigl( \frac4\b {\sum_{s\in \{\pm\}} \int_{D \ti [0, \b]} |\nabla V_t( (x_s,y) , \CC, \La)|^2 \d (y, t) \Bigr)^{1/2} } \\
			+{\sum_{s \in \{\pm\}}\#\big\{b\in \CC:\, \inf_{t \le \b} |b_x(t) - x_s| \le 1\big\}}.
	\end{multline*}
	\enen
\end{lemma}

\begin{proof}
	We start with part (1).
	Gauss's theorem yields that for all $t \le \b$ 
	\begin{align}
		\label{charge_imb_eq}
		\lim_{u\to-\ff}\int_D \nabla V_t((u, y), \CC, \La) \cdot(- \vec e_x)\, \d y + \int_D \nabla V_t((x_-, y), \CC, \La) \cdot \vec e_x\, \d y = \mu_t(\CC, (-\ff, x_-]\ti D).
	\end{align}
	The first term on the left-hand side vanishes since the electric field at $-\ff$ vanishes, {which is easily proven with~\eqref{kernel_expansion}, charge neutrality, and properties of purely one-dimensional systems \cite{aizenman1980structure}}.  
	Hence, by Cauchy-Schwarz,
	\begin{align} \label{eq:chargimb1}
		\mu_t\big(\CC, (-\ff, x_-]\ti D)\big)^2  \le \Big(\int_D |\nabla V_t((x_-, y), \CC, \La)| \, \d y\Big)^2 
		\le \int_D |\nabla V_t((x_-, y), \CC, \La)|^2 \, \d y.
	\end{align}
	Now note that for every $t \le \b$, the charge imbalance is at least
	$$\big|\mu_t(\CC, (-\ff, x_-]\ti D)\big| \ge \big|\mu_0(\CC, {(-\ff, x_-]\ti D}) \big| - \#\big\{b\in \CC:\, \inf_{t \le \b} |b_x(t) - x_-| \le 1\big\}.$$
	Hence, inserting this bound into \eqref{eq:chargimb1} and integrating over time concludes the proof.

	The arguments for part (2) are very similar. Gauss's theorem yields that for all $t \le \b$ 
	\begin{align*}
		{\int_{D} \nabla V_t((x_+,y), \CC, \La) \cdot \vec e_x\, \d y 
				+ \int_{D} \nabla V_t((x_-,y), \CC, \La) \cdot(- \vec e_x)\, \d y} 
			= \mu_t(\CC, \La'). 
	\end{align*}
	Hence, by Cauchy-Schwarz and the fact that $|\pa \La'|=2$ 
	\begin{align} \label{eq:chargimb2}
	 \mu_t(\CC, \La')^2 \le { \(  \sum_{s\in \{+,-\}} \int_{D} |\nabla V_t(( x_s,y), \CC, \La)| \, \d y\)^2 
		 \le 4  \sum_{s\in \{+,-\}} \int_{D} |\nabla V_t(( x_s,y), \CC, \La)|^2 \, \d y}
	\end{align}
	and we now conclude as above.
\end{proof}

\begin{definition}[$\e$-dense abscissas]\label{def:dense abs}
	Let $\e\in(0,1/2)$ and set $\xi = \e^{-3}>0$.
For $\CC \in \Conf(\La_R)$, the set $\ms{Dense}_\e(\CC)$ of \emph{$\e$-dense abscissas} consists of all 
	 $$x_0\in \big[2R^{7/8}, R - 2R^{7/8}\big]$$
	 such that
\be \label{eq:birreg}
		\#\{b\in \CC_{\La_R}:\, |b_x(0) - x_0|\le \xi \} \ge {\frac{4\xi}\e}.
\ee
\end{definition}
 We bound the Lebesgue measure of dense abscissas in terms of the energy. 
\begin{lemma}[A bound on $\e$-dense sets] \label{lem:edense}
	Let $M>2$, $\e\in(0,1/M^2)$, and $R> (2/\e)^8$ and let $\CC \in \Conf(\La_R)$ be such that 
	$$\frac1{2\b}\int_{\La_R\ti D} |	\nabla V_t(z,\CC,\La_R)|^2 \d z < M R$$ 
	and every bridge in $\CC$ has a range less than $R^{7/8}$. Then,
	\be\label{eq:constM}
		|\ms{Dense}_\e(\CC)| \le \e R.
	\ee
\end{lemma} 

Let us remark that the choices of parameters $M$, $\e$, and $R$ are not optimal, but rather, chosen to be consistent with assumptions that we will use later on.

\begin{proof} 
	We proceed by contradiction. {We first claim that if \eqref{eq:constM} does not hold, then 
	\begin{align}
		\label{laDenseEq}
{\#\CC_\La\geq 2R}
	\end{align}
	where $\La = [\tfrac32R^{7/8}, R - \tfrac32R^{7/8}]\ti D$. If we have \eqref{laDenseEq},} then
	$$\sup_{t \le \b}\sup_{x \le \tfrac12 R^{7/8}}\mu_t(\CC, [x, R - x] \ti D) \leq -R.$$ 
	On the other hand, by \eqref{eq:chargimb2},
	\begin{align*}
	\mu_t(\CC, [x, R - x] \ti D)^2 &\le 4 \int_D \bigl(|\nabla V_t(x,y)|^2 +|\nabla V_t(R-x,y)|^2\bigr) \d y. 
	\end{align*}
	After integrating over $[0, \frac12 R^{7/8}] \ti [0, \b]$, we arrive at 
	$$\frac1{2\b} \int_{[0, \frac12 R^{7/8}] \ti [0, \b]}\mu_t(\CC, [x, R - x] \ti D)^2 \d (x, t) < 4MR.$$
 But $R> 8M> 16$, so we obtain the contradiction 
	$$\frac14R^{7/8}\cdot R^2 < 4MR \le {\frac12 R^2}.$$

	It remains to show that if \eqref{eq:constM} does not hold, then {$\#\CC_\La\ge 2R.$} To prove this assertion, we integrate \eqref{eq:birreg} over all $x_0 \in \ms{Dense}_\e(\CC)$ to obtain that 
	\begin{align*}
		4\e^{-4}|\ms{Dense}_\e(\CC)| &\le \int_{{[2R^{7/8}, R - 2R^{7/8}]}} \#\{b\in \CC_{\La_R}:\, |b_x(0) - x_0|\le \e^{-3} \} \d x_0 \\
		&=\sum_{b\in \CC_{\La_R}}\int_{{[2R^{7/8}, R - 2R^{7/8}]}} \one\{|b_x(0) - x_0|\le \e^{-3} \} \d x_0\\
		&\le 		2\e^{-3}\#\CC_\La.
	\end{align*}
	Hence, dividing both sides by $2\e^{-3}$ concludes the proof.
\end{proof} 

\section{Screening} \label{sec:screening}

Screening starts with the following: {with every configuration $\CC\in \Conf(\La_R)$} we will associate a new configuration $\Cs \in \Neut(\La_R)$ and an electric field 
\begin{align*}
E^\scr: \La_R\ti [0, \b]&\to \R^{k+1}\\
(z, t)&\mapsto E_t^\scr(z)
\end{align*}
such that 
\begin{itemize}
	\item [(i)]For most $\CC$, the configuration $\Cs$ differs from $\CC$ only near the boundaries of $\La_R$. 
	A precise description of $\Cs$ is in Subsection \ref{subsec:screenconfig}.
	\item [(ii)] The field $E^\scr$ is compatible with $\Cs$ and {a background charge distribution of Lebesgue measure in $\Lambda_R$, in other words, in the region $\Lambda_R$ it satisfies}  
				\be \label{eq:compatible}
					\nabla \cdot E_t^\scr = - \sum_{b\in \Cs} \de_{b(t)}^\eta + \one_{\La_R}
				\ee
			for all $t \le \b$. 
	\item [(iii)]The field $E^\scr$ is \emph{screened}, i.e., 
				\be \label{eq:screened} 
					{E_t^\scr(0, y) \cdot \vec e_x = E_t^\scr(R, y) \cdot \vec e_x = 0}
				\ee
				for all $y\in D$ and $t \le \b$.
	\item [(iv)] The field satisfies an energy bound of the type 
			\be
				\frac1{2\b}\int_{\La_R\ti [0, \b]} \bigl|E_t^\scr(z)\bigr|^2 \d (z, t) \le \tH_{\La_R}(\CC) + \text{ error term}. 
			\ee
			{The error term will be made precise in Proposition \ref{prop:detscreening}.}
\end{itemize} 

Even when restricting to configurations close to a well-behaved target process $P$, some of the bridges in the configurations may fluctuate too wildly for the envisioned screening construction. Thus, we need a preprocessing step. 
Recall that the $x$-range of $b$ is $\psi(b) = \sup_{t \le \b} |b_x(t) - b_x(0)|$ and remember Lemma \ref{lem:entropy-estim}. 

\begin{definition}[$(M, \e)$-regularity] \label{def:regularity}
Let {$M > 2$, $\e\in(0, 1/M^2)$ and	$R> (2/\e)^8$.}	A bridge configuration $\CC \in \Conf( \La_R)$ is \emph{$(M, \e)$-regular} if {\mbox{$\tH_{\La_R}(\CC) +\e < MR$}} and 
	\be
	\label{eRegDef} 
		\sum_{b\in \CC_{\La_R}} \big(\psi(b)^{7/6} - \e^{-7/3}\big)_+ < {\frac{\e R}2 }. 
	\ee
\end{definition} 

\begin{remark}[Choice of parameters]\label{rem:reasonable choices}
As already mentioned, we have chosen $\e$ and $R$ according to later purposes. We however note that $R >(2/\e)^8$ is large enough so that condition \eqref{eRegDef} forbids any single bridge with $x$-range bigger or equal to $R^{7/8}$ since then we would have the contradiction
$$R^{49/48} - \e^{-7/3}<\e R/2.$$
Thus, Lemma \ref{lem:edense} holds for all $(M, \e)$-regular configurations.
\end{remark}

\begin{figure}[!htpb]
	\input{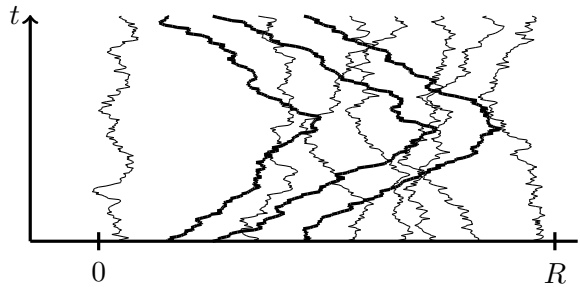}
	\caption{{One-dimensional projection of} a non-regular configuration}
	\label{regFig}
\end{figure}

For $(M, \e)$-regular $\CC\in \Conf(\La_R)$, by definition of the energy and the regularity, there exists a domain $\La \supset \La_R$, and a configuration $\t\CC\in \Neut(\La)$ 
with $\t\CC_{\La_R} = \CC$, such that 
\be \label{eq:bcchoice}
	\frac1{2\b}\int_{\La_R\ti [0, \b]} |\nabla V_t(z, \t\CC, \La)|^2 \d(z, t) \le  \tH_{\La_R}(\CC) +\e < MR.
\ee
We choose such $\La, \t\CC$ and set 
\be \label{def:Echoice} 
	\wt E_t(z):= \nabla V_t(z, \t\CC, \La). 
\ee

\subsection{Choice of a good boundary abscissa $x_0$}\label{subsec:boundary}

The construction of $\Cs$ depends on a choice of boundary abscissa which we denote as $x_0$. In the vein of the classical case from \cite[Section 5]{serfInv}, the screening construction modifies the configuration in boundary layers corresponding roughly to $[0, x_0]\ti D$ and $[R - x_0, R]\ti D$. \\ 

The boundary $x_0$ is chosen such that some electric field average is not too high. In addition, we ask that only few bridges cross the boundaries of $I_0 := [x_0, R - x_0]$. 

\begin{definition}[$\e$-tame abscissas] \label{def:etame} Let $\e \in (0, 1/M^2)$ and  $R>(2/\e)^8$.
	Then, \tcb{setting $$\xi:= \e^{-3},$$} we say that \mbox{$x_0 \in [\e R, R - \e R]$} is \emph{$\e$-tame} if 
 	\begin{enumerate}
 		\item $x_0$ is not $\e$-dense and
		\item if $b \in \CC$ is such that $\inf_{t \le \b}|b_x(t) - x_0| \le 1$, then $|b_x(0) -x_0|\le \xi$.
 	\end{enumerate}
\end{definition} 

{Notice that $R>(2/\e)^8$ implies $\e R> R^{7/8}$, therefore item (1) of the definition says that the number of bridges that have their initial point $b_x(0)$ within distance $\xi$ of $x_0$ is smaller or equal to $4 \xi /\e$.}

{The next lemma will be needed later to prove Lemma \ref{noFracLem} and Proposition \ref{prop:detscreening}.}

\begin{lemma}[Existence of good boundaries]
	\label{goodBoundLem}
	Let {$\e\in(0,1/M^2)$, and $R >(2/\e)^8$}. Furthermore, let {$I\ti D \su \La_{R/2}$ be such that {$I\su[\e R, R/2]$} is} an interval of length at least $4\e R$. If $\CC$ is $(M, \e)$-regular, then there exists a  point $x_0(\CC) \in I$ such that $x_0$ and $R - x_0$ are $\e$-tame, and the following bounds hold: 
	\be \label{eq:gb1} 
		{\int_{D\times [0,\beta]} \Bigl( |\wt E_t(x_0,y)|^2 + |\wt E_t(R-x_0,y)|^2\Bigr) \d(y,t)}
	\le \frac1{\e R}\int_{\La_R \ti [0, \b]} |\wt E_t(z)|^2 \d (z, t),
	\ee
	and
	\be \label{eq:gb2} 
	 \int_{J_{x_0, R} 		\ti \D \ti [0, \b]}|\wt E_t(z)|^2 \d(z, t) \le 
	\frac1{\e \sqrt R}\int_{\La_R \ti [0, \b]} |\wt E_t(z)|^2 \d (z, t)
	\ee
	where 
	$$J_{x_0, R} = [x_0, x_0 + \sqrt R] \cup [R - x_0 - \sqrt R, R - x_0].$$
\end{lemma}

\begin{proof}
	We bound the Lebesgue measure of the set of $x_0$'s violating the required conditions. 
		First, the set $A_1\su I$ of abscissas violating condition \eqref{eq:gb1} has Lebesgue measure at most $\e R$. Indeed,
	abbreviate 
	$$
	C = \frac1{2 \b R}\int_{\La_R \ti [0, \b]} |\wt E_t(z)|^2 \d (z, t).
	$$
	For $x\in [0, R/2)$, let 
	$$
	q(x)= \frac1{2\b} {\int_{D\ti [0, \b]} \Bigl(\bigl| \wt E_t(x,y)\bigr|^2 + \bigl| \wt E_t(R-x,y)\bigr|^2 \Bigr) \d (y, t)}.
	$$
	Since $I\su[0, R/2]$, 
	$$
	\big| \{x\in I:\, q(x) > C/\e \}\big| \le \frac \e{C} \int_0^{R/2} q(x)\d x = \frac{\e }{2\b C} \int_{\La_R \ti [0, \b]} |\wt E_t(z)|^2 \d(z, t) = \e R. 
	$$
	Second, the set $A_2\su I$ of $x_0$'s violating condition \eqref{eq:gb2} has Lebesgue measure at most $\e R$. Indeed, let $q_2(x_0)$ be the left-hand side of \eqref{eq:gb2}. By Fubini's theorem,
	$$
	\int_0^{R/2} q_2(x_0) \d x_0 \le {CR}
	$$
	hence 
	$$
	\big| \{x_0\in I:\, q_2(x_0) > C/\e \}\big| \le \e R.
	$$
	Finally, we assert that for the set $A_3$ of $x_0$'s that are not tame, $|A_3|<2 \e R$. Then, we conclude the proof of the lemma by combining the above bounds, as the measure of $A_1\cup A_2 \cup A_3$ is strictly less than $4\e R$ and since $I$ is of length $4\e R$ it is non-empty. To prove the assertion, note that 
	since $\CC$ is $(M, \e)$-regular, we may use Lemma \ref{lem:edense} to bound the Lebesgue measure of the set of $\e$-dense abscissas by $\e R$. Next, let 
		\begin{align}
			T_r(x_0) = [x_0 - r, x_0 + r] \ti D
		\end{align}
		be the product of an interval of length $2r$ centered at {$x_0 $} with the domain $D$.
		 We bound the measure of abscissas satisfying condition \eqref{eRegDef} but violating the second tameness condition by
		\begin{align}
			\sum_{b \in \CC}\big|[b_x(0) - \psi(b), b_x(0) + \psi(b)] \setminus T_{\xi - 1}(b_x(0))\big| \le 2\sum_{b \in \CC}(\psi(b) - \e^{-2})_+.
		\end{align}
		Hence, invoking the bound \eqref{eRegDef} shows $|A_3|<2 \e R$.
\end{proof}

\subsection{Construction of screened configurations $\Cs$.}\label{subsec:screenconfig}

{Lemma \ref{goodBoundLem} applies to the interval $[6\e R, 10\e R]$ which we will use throughout this section. In particular, we will assume $x_0$ designates a choice of the interface with this choice of interval. (A different application with the choice $[2\e R,6\e R]$ is used following Eq.~\eqref{eq:x0prime}, and it will be convenient for our choice of the interval in this section to be to the right of $[2\e R,6\e R]$; this is the reason for setting the left endpoint to $6\e R$.)
 Having chosen a good boundary $x_0$, we proceed with the definition of $\Cs$.

If $\CC\in \Conf(\La_R)$ is not $(M, \e)$-regular, then 
$\Cs$ consists of time-constant bridges (straight lines) with initial points 
\be\label{screeningirreg}
	\big(i + \frac12, 0\big),\qquad {i=0,\dots, R-1}. 
\ee  
 For $(M, \e)$-regular $\CC\in\Conf(\La_R)$,
in a preliminary regularization step, we rearrange all bridges crossing the boundaries $\pa I_0\ti D$, {with $I_0 = [x_0, R-x_0]$ with $x_0$ chosen above. By the choice of $x_0$, all these bridges have their starting point within $T_\xi(x_0)$. We remove all bridges in $\CC$ with starting point in $T_\xi(x_0)$ and replace them with the same number of bridges inside $[x_0 + 1/2, x_0 + 1] \ti \D\su T_\xi(x_0)$. 
The new bridges are time-constant straight lines with regularly-spaced starting points $(x_i, 0)$, {$i=1,\dots, n$}  
	where $n =\#\CC_{T_\xi(x_0)} $ and $x_i = x_0 + \frac12 + \frac1{4n} + (i - 1) \frac1{2n}$. 
	We proceed similarly in $T_\xi(R -x_0)$. 
	
The newly obtained 
\begin{align}\label{eq:regstep}
\textit{ regularized configuration}\quad \CC' \in \Conf(\La_R)
\end{align}
coincides with $\CC$ except possibly in $T_\xi(x_0)\cup T_\xi(R -x_0).$ Moreover, the regularized configuration $\CC'$ has no crossings of the interface at $x_0$ and $R-x_0$. 

Changing $\t\CC$ from \eqref{eq:bcchoice} similarly results in a configuration $\t\CC' \in \Neut(\La)$ with projection $\t\CC'_{\La_R} ={\CC}'$ and electric field 
\be\label{def:E'}
		\wt E'_t(z) := \nabla V_t(z, \t\CC', \La).
\ee
We split the domain $[0, R/2] \ti D$ into sub-domains at the interfaces
$$
	0 < x_- < x_0 < x_+ < R/2,
$$
where 
\be\label{rmmdef}\rmm := x_0 + \xi\ee and $x_-$ is specified as follows. Recall from \eqref{eq:bcchoice}-\eqref{def:Echoice} our choice of $\La$ 
which we now write as $\La = [L_-, L_+] \ti D$ and let 
\be \label{eq:d0}
d_0:= \#\tilde \CC'_{[L_-, x_0]\ti D} - (x_0 - L_-)
\ee
be the net {negative} charge (or deficit of charge) to the left of $x_0$ for the extended regularized configuration $\tilde \CC'\in \Neut(\La)$.
Note that $d_0 + x_0 \in \Z$ since $L_-\in \Z$ (it should be noted that if this is not an integer, then one runs into an issue with fractional charges---see the remark following \eqref{def:EMP}).
In Lemma \ref{noFracLem} below, we check that $x_0 - 2|d_0|\ge 0$ for sufficiently large $R$, which we use later in (\ref{2d_0 choice}). Set 
\be \label{eq:xminusdef}
\rpp := \lfloor x_0 - 2|d_0| \rfloor \in {\N_0}.
\ee
\tcb{Here, we emphasize that $x_0$, and thus $\rmm$ and $\rpp$, depend on the original configuration $\CC$.}

\begin{definition}[$\Css$ for $(M,\e)$-regular configurations]\label{def:css}
The new configuration $\Css$ is defined separately in each subdomain. Then, we proceed in the domain $[R/2, R] \ti D$ similarly.
\begin{enumerate}
	\item In the central subdomain 
	\be\label{def:lambdaplus}
		\La_+ := [\rmm, R - \rmm] \ti D, 
	\ee
	 the configuration $\Css$ coincides with $\CC$, i.e., $\Css_{\La_+} := \CC_{\La_+}$. 
	\item In $[x_0, x_+]\ti D$, the configuration $\Css$ coincides with the time-constant bridges as rearranged in the regularization step above. That is, $\Css$ and $\CC'$ have the same restrictions to $[x_0, x_+]\ti D$. 
	\item In $[x_-, x_0)\ti D$, we place 
		$$a_+ := x_0 - \rpp + d_0$$ 
		                time-constant bridges with regularly spaced starting points $(x_i, 0)$ where $x_i = x_- + i (x_+ - x_-) / a_+ $, $i = 1, \dots, a_+$. Notice that $x_0 + d_0 - \rpp$ is integer because $x_0 + d_0$ and $x_-$ are, and $x_0+d_0 - \rpp\ge 0$ because $\rpp\le x_0 - 2 |d_0|\le x_0 + d_0$.  
	\item In $[0, x_-)\ti D$, we place $\rpp$ time-constant bridges with regularly-spaced starting points $(i - \frac12, 0)$, $i = 1, \dots, \rpp$, so that in this region, $\Cs$ is charge neutral.
\end{enumerate}
\end{definition}

\begin{figure}[!htpb]
	\input{preScreenConf}
	\input{postScreenConf}
	\caption{One-dimensional projection of a transformation of $\CC$ (left) into $\Css$ (right).}
	\label{screenConfFig}
\end{figure}

\tcb{We note that the construction of $\Cs$ is robust in the sense that we could replace $x_0$ by a rational value sufficiently close to a good boundary, thereby making $\Cs$ measurable in the input data $\CC$.}

\begin{lemma}
	\label{noFracLem}
	{Let $M > 2$, $\e\in(0,1/M^2)$, and $R >(2/\e)^8$}. Assume that $\CC\in \Conf(\La_R)$ is $(M, \e)$-regular. Then for any tame abscissa {$x_0\in  [6\e R, 10\e R]$}, with $d_0 \in \R$ as in \eqref{eq:d0},
		 $$x_0 \ge 2|d_0|.
	$$
\end{lemma}

\begin{proof}
	Let 
\be
d:= \#\tilde \CC_{[L_-, x_0]\ti D} - (x_0 - L_-)
\ee
	 be defined in the same way as $d_0$ in \eqref{eq:d0} but with $\tilde \CC'$ replaced by the non-regularized $\t\CC$ and note that by tameness of $x_0$,
	\be
		|d-d_0|\le  4\e^{-4}.
	\ee
	Now, by Lemma \ref{lem:chargimb1},
	\be
	\label{noFracEq}
	{|d|} \le \Bigl( \frac1\b\int_{\D\ti [0, \b]} |\wt E_t(x_0, y)|^2 \d (y, t) \Bigr)^{1/2}
	+ \#\big\{b\in \t\CC:\, \inf_{t \le \b}\ |b_x(t) - x_0| \le 1 \big\}.
	\ee
      The inequality \eqref{eq:gb1} from Lemma \ref{goodBoundLem} yields
	\be
	\frac1\b\int_{\D\ti [0, \b]} |\wt E_t(x_0, y)|^2 \d (y, t)
	\le  \frac2{2\e \b R}\int_{\La_R \ti [0, \b]} |\wt E_t(z)|^2 \d (z, t) {\le \frac {2M}{\e} \le 2\e^{-3/2}},
	\ee
	Since $x_0$ is $\e$-tame the cardinality in \eqref{noFracEq} is bounded by $4\e^{-4}$, so that we arrive at 
	\be
	{|d|\le 5\e^{-4}}.
	\ee
	As {$R >(2/\e)^8$} and {$x_0 \in [6\e R, 10\e R]$}, it follows that
	\be \label{eq:d0bound}
	|d_0|   \le 9\e^{-4} < \frac{\e R}2 \le \frac{x_0}2.
	\ee
\end{proof}

\subsection{Construction of screened electric fields $E_t^\scr$.}
We now construct the screened electric field $E^{\scr}_t$, along the lines of \cite[Proposition 5.1]{serfInv}, which is compatible with $\Css$ in the sense of \eqref{eq:compatible}. We define $\Es_t$ separately in each of the domains $[0, \rpp] \ti D$, $[\rpp, x_0] \ti D$, and $[x_0, R/2] \ti D$.

\vspace{3mm}
\noindent{\it Construction in ${[x_0, R/2]} \ti D$.}
We put
$\Es_t(\cdot) = \wt E'_t(\cdot)$ defined in \eqref{def:E'}, noting that in this region, 
 this electric field is	compatible with the configuration $\Css$ defined in Subsection \ref{subsec:screenconfig}.

\vspace{3mm}
\noindent{\it Construction in $[0, \rpp] \ti D$.}
Since the bridges are nicely separated, we define $\Es_t$ inside this domain by first constructing the electric field in the cell $[i-1,i]\ti D$ around each time-constant bridge $b_i$ at $x_i$ and then pasting the separate pieces together.
 In particular, for
 	\be \label{eq:et2} 
 	\div\big(\Es_t(x, y)\big) = - \de^\eta_{x_i} + 1,
 	\ee
 	with $\Es_t(x, y) \cdot \vec e_x = 0$ on the boundaries,
 	a necessary and sufficient condition for the existence of a  solution is charge neutrality, i.e., 
	that the integral of the right-hand side over $[i-1,i]\ti D$ vanishes \cite[Theorem 1.2]{neumann}.

Pasting together the fields defines the screened electric field $\Es_t$ on $[0, \rpp] \ti D$. \tcb{This pasting may lead to vector fields that are discontinuous across an interface so that  just as in \cite[Section 5.1]{serfInv}, we need to work with distributional divergences. From \cite[Section 5.1]{serfInv}, we also recall that no additional divergences are created through the pasting operation since the normal derivatives along the interface coincide.}

\vspace{3mm}
\noindent{\it Construction in $[\rpp, x_0] \ti \D$.}
Defining the electric field in $[\rpp, x_0] \ti \D$ is the core of the screening construction. Indeed, here the task is to come up with a compatible electric field whose normal components agree with those of the field $\Es_t(\cdot)$ from the previous step at the interface $\{x_0\} \ti \D$, and vanish at the interface $\{x_-\} \ti \D$. 

For this purpose, we subdivide the domain $[\rpp, x_0] \ti \D$ into $a_+ = x_0 - \rpp + d_0$ congruent domains of the form $\Ii \ti \D$, 
\be\label{eq:ell check}
	\Ii = \bigl[ x_0 + (i+\tfrac12) \ell, x_0 + (i + 1 + \tfrac12) \ell\bigr], \quad 
	\ell = \frac{x_0 - x_-}{a_+}, \quad i=0,\dots, a_+- 1,
\ee
 each centered at a time-constant bridge $b_i$. Note that the choice of $\rpp$ in \eqref{eq:xminusdef} implies that \be \label{2d_0 choice}
 \frac23 \le |\Ii|=\ell \le 2.
 \ee
 
We think of the system in $[x_-, x_0]\ti D$ as a superposition of two systems, a neutral system with a background of $y$-dependent charge density together with the time-constant bridges from $\Cs$, and an additional system with no particles and $y$-dependent charged background only. Accordingly the field is defined as a sum of two contributions.

Let us start with the latter contribution coming from a system with no particles. This has only components in the $x$-direction, and the $x$-component interpolates linearly between $0$ and $\Es_t(x_0, y)$: we define
\begin{align}\label{def:E1}
	E^{(1)}_t(x, y) := \frac{x - \rpp}{x_0 - \rpp}\, ( \Es_t(x_0, y) \cdot \vec e_x)\, \vec e_x. 
\end{align}
Notice 
\begin{align} \label{eq:et1}
		\div\big(E^{(1)}_t(x, y)\big) = \frac{\Es_t(x_0, y) \cdot \vec e_x}{x_0 - \rpp} = :\rho_t^{(1)}(y), 
\end{align}
thus $E_t^{(1)}$ is compatible with a charged background of $y$-dependent charge density given by the right-hand side of \eqref{eq:et1}. 

For the second contribution, define $E_t^{(2, i)}$ as the solution in $\Ii \ti D$ of 
\be \label{eq:et22} 
	\div\big(E^{(2, i)}_t(x, y)\big) = - \de^\eta_{b_i(0)} + 1 - \rho_t^{(1)}(y)
\ee
with {periodic boundary conditions in the $y$-direction} and $E^{(2, i)}_t(x, y) \cdot \vec e_x = 0$ for $x\in \pa \Ii$. As in \eqref{eq:et2}, a necessary and sufficient condition for the existence of a solution is that the integral of the right-hand side over the domain $\Ii$ vanishes, i.e., that $\Ii$ is charge-neutral. We compute 
\be
\begin{aligned}
	\int_{\Ii\ti D} \bigl( 1 - \rho_t^{(1)}(y)\bigr) \d (x, y) & = |\Ii| - \frac{|\Ii|}{x_0 - x_-}\, \int_D\Es_t(x_0, y)\cdot \vec e_x\, \dd y \\
	 & = |\Ii|\Bigl( 1 - \frac{- d_0}{x_0 - x_-}\Bigr)\\
	 & = \Bigl( \frac{x_0 - x_-}{a_+}\Bigr) \Bigl(1 - \frac{- d_0}{x_0 - x_-}\Bigr)= 1.
\end{aligned} 
\ee
 {Since $|\Ii| \geq 2/3 > 2\eta$} (see the remark after \eqref{eq:ell check} and the assumption on $\eta$ in Section \ref{sec:truncation}), we see that indeed the cell $\Ii\ti D$ with background of charge density $1 - \rho_t^{(1)}(y)$ and smeared charge $\de_{b_i(0)}^\eta$ is neutral. 

Finally we set 
\be
	E^{(2)}_t = \sum_i E^{(2, i)}_t \one_{\Ii\ti D}, \quad \Es_t = E^{(1)}_t + E^{(2)}_t
\ee
in $[x_-, x_0]\ti D$ and note that the field $\Es_t$ is compatible with $\Css$ in that domain.

Altogether, we obtain field $\Es$ in $[0, R]\ti D$ with the right properties. The combined construction is illustrated in Figure \ref{screenFieldFig}.
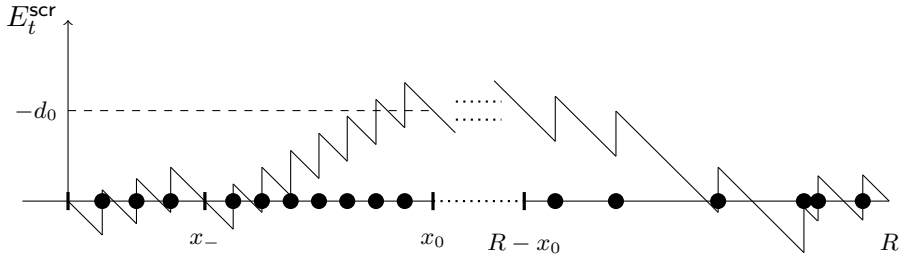
\begin{figure}[!htpb]
	\begin{tikzpicture}[yscale = .6, xscale =-.6]

	%BETA-AXIS
	\draw[->] (18,0)--(18,4);
	\coordinate[label=180: $\Es_t$] (D) at (18,4);

	%X-AXIS
	\draw[] (0,0)--(8,0);
	\draw[] (19,0)--(10,0);
	\draw[thick,dotted] (10,0)--(8,0);
	\draw[thick,dotted] (9.5,2.2)--(8.5,2.2);
	\draw[thick,dotted] (9.5,1.8)--(8.5,1.8);

	%D0
	\draw[dashed] (18,2) -- (10,2);
	\coordinate[label=180: {\footnotesize$-d_0$}] (D) at (18,2);

	%TICKS
	\draw[very thick] (18,-.2)--(18,.2);
	\draw[very thick] (10,-.2)--(10,.2);
	\draw[very thick] (8,-.2)--(8,.2);
	\draw[very thick] (15,-.2)--(15,.2);

	%LABELS

	\coordinate[label=-90: {\footnotesize$R$}] (D) at (0,-.5);
	\coordinate[label=-90: {\footnotesize$x_-$}] (D) at (15,-.5);
	\coordinate[label=-90: {\footnotesize$ x_0$}] (D) at (10,-.5);
	\coordinate[label=-90: {\footnotesize$ R - x_0$}] (D) at (8,-.5);

	%ELECTRIC FIELD
	\draw (0.00, 0.00)--(0.58, 0.58);
	\fill (0.58,0) circle (5pt);
	\draw (0.58, 0.58)--(0.58, -0.42);

	\draw (0.58, -0.42)--(1.56, 0.56);
	\fill (1.56,0) circle (5pt);
	\draw (1.56, 0.56)--(1.56, -0.44);

	\draw (1.57, -0.44)--(1.87, -0.14);
	\fill (1.87,0) circle (5pt);
	\draw (1.87, -0.14)--(1.87, -1.14);

	\draw (1.87, -1.14)--(3.75, 0.75);
	\fill (3.75,0) circle (5pt);
	\draw (3.75, 0.75)--(3.75, -0.25);

	\draw (3.75, -0.25)--(5.99, 1.99);
	\fill (5.99,0) circle (5pt);
	\draw (5.99, 1.99)--(5.99, 0.99);

	\draw (5.99, 0.99)--(7.32, 2.32);
	\fill (7.32,0) circle (5pt);
	\draw (7.32, 2.32)--(7.32, 1.32);

	\draw (7.32, 1.32)--(8.66, 2.66);
%	\fill (8.66,0) circle (5pt);
%	\draw (8.66, 2.66)--(8.66, 1.66);

%	\draw (8.66, 1.66)--(9.51, 2.51);
%	\fill (9.51,0) circle (5pt);
%	\draw (9.51, 2.51)--(9.51, 1.51);

	\draw (9.51, 1.51)--(10.62, 2.62);
	\fill (10.62,0) circle (5pt);
	\draw (10.62, 2.62)--(10.62, 1.62);

\draw (10.62, 1.62)--(11.25, 2.25);
\fill (11.25,0) circle (5pt);
\draw (11.25, 2.25)--(11.25, 1.25);

\draw (11.25, 1.25)--(11.88, 1.88);
\fill (11.88,0) circle (5pt);
\draw (11.88, 1.88)--(11.88, 0.88);

\draw (11.88, 0.88)--(12.50, 1.50);
\fill (12.50,0) circle (5pt);
\draw (12.50, 1.50)--(12.50, 0.50);

\draw (12.50, 0.50)--(13.12, 1.12);
\fill (13.12,0) circle (5pt);
\draw (13.12, 1.12)--(13.12, 0.12);

\draw (13.12, 0.12)--(13.75, 0.75);
\fill (13.75,0) circle (5pt);
\draw (13.75, 0.75)--(13.75, -0.25);

\draw (13.75, -0.25)--(14.38, 0.38);
\fill (14.38,0) circle (5pt);
\draw (14.38, 0.38)--(14.38, -0.62);

\draw (14.38, -0.62)--(15.75, 0.75);
\fill (15.75,0) circle (5pt);
\draw (15.75, 0.75)--(15.75, -0.25);

\draw (15.75, -0.25)--(16.50, 0.50);
\fill (16.50,0) circle (5pt);
\draw (16.50, 0.50)--(16.50, -0.50);

\draw (16.50, -0.50)--(17.25, 0.25);
\fill (17.25,0) circle (5pt);
\draw (17.25, 0.25)--(17.25, -0.75);

\draw (17.25, -0.75)--(18.00, 0.00);
\end{tikzpicture}
	\caption{One-dimensional projection of the electric field $\Es_t$ with positions of charges marked on $x$-axis.}
	\label{screenFieldFig}
\end{figure}

\subsection{Energy change due to screening electric fields}

{Using similar techniques to those employed in \cite{serfInv, aizenman2010symmetry}, we will verify that the so-constructed screened electric field has an energy that does not exceed the original one by much. To this end we first need a lemma discussing the effect on the energy change due to moving a fixed number of charges.}
{Let $$g^\eta(z) := \int_\T \gN(z - z' )\de^\eta_o(\d z')$$ denote the Green's function associated with a smeared charge.}

\begin{lemma}[Moving charges]
	\label{diff1Lem}
	Let $z_1, \dots, z_n\in \La_R$ and $E:S\to \R^{k+1}$ be such that $-\div(E) = \sum_{i \le n} \de_{z_i}^{(\eta)} - \one_{\La_R}$. For $z'_1\in \La_R$ and $z\in S$, define 
	$$E'(z) := E(z) + \nabla g^\eta(z - z'_1) - \nabla g^\eta(z - z_1).$$
	Then, $-\div(E') = \de_{z'_1}^{(\eta )}+ \sum_{i \ge 2} \de^{(\eta)}_{z_i} - \one_{\La_R}$. In addition there exists a constant $c = c(D, \eta) > 0$ such that $\sup_{z \in \La_R}|E'(z) - E(z)| \le c$. Moreover, for any $\a > 0$ there exist $\cm, \cm' > 0$ such that 
	$$\big|E'(z) - E( z)\big| \le \cm\exp\big(- \cm'(|x - x_1| \wedge |x - x'_1|)\big).$$
	holds whenever $|x_1 - x'_1| \le \a$. 
\end{lemma}
\begin{proof}
	The compatibility of $E'$ is an immediate consequence of the definition. Moreover, since smearing removes the singular part of the Green's function $g$, the increment 
	$$E'(z) - E(z) = \nabla g^\eta(z - z'_1) - \nabla g^\eta(z - z_1)$$
	remains bounded. For the final claim, we fix $\a > 0$ and assume that $|x_1 - x'_1| \le \a$. The final claim then follows from \eqref{kernel_expansion}.
\end{proof}

\begin{proposition}[Energy change from screening electric fields] \label{prop:detscreening}
	Let $M > 2$, {$\e\in(0,1/M^2)$, $R > (2/\e)^8$} and $\CC\in \Conf(\La_R)$ be $(M, \e)$-regular. Then,  the configuration $\Cs\in \Neut(\La_R)$ satisfies
	\begin{itemize}
	\item [(a)] $\Cs$ and $\CC$ coincide on {$[10\e R, R - 10\e R]\ti D$, i.e., 
		$$ \CC_{[10\e R, R - 10\e R]\ti D} = \Cs_{[10\e R, R - 10\e R]\ti D}.$$}
	\item [(b)] The field $E^\scr$ is compatible with $\Cs$ in $\La_R$, i.e., it satisfies Eq. \eqref{eq:compatible}. 
		\item [(c)] The field $E^\scr$ is screened, i.e., it satisfies \eqref{eq:screened}. 
	\item [(d)] For every $t \le \b$, the screened field has an electrostatic energy bounded as 
		$$\int_{\La_R} \bigl|E_t^\scr(z)\bigr|^2 \d z \le \vp_0(\wt E_t, \e, R),$$
			where $\vp_0(\wt E_t, \e, R) \ge 0$ {is measurable in $t$} and satisfies

		$$\frac1{2R\b}		 \int_{[0, \b]}\vp_0(\wt E_t, \e, R) \d t
				\le \frac1{2R \b} \int_{\La_R \ti [0, \b]} |\wt E_t(z)|^2 \d (z, t) 
				+ \vp(\e, R)M$$
				where  $\vp(\e, R) \ge 0$  satisfies
			\be \label{phiBoundEq}
					\limsup_{\e\to 0} \limsup_{R\to \ff}\vp(\e, R) = 0.
			\ee
	\end{itemize} 
\end{proposition} 

\begin{proof}
		Parts (a), (b), and (c) hold true by definition of $\Cs$ and $\Es_t$. 
		
		It remains to verify the energy condition (d). 
	
\vspace{3mm}\noindent	{\it Energy in ${[0, x_-]} \ti D$.}
	{Note that the energy each of the charge-neutral unit-volume constituent fields of $\wt E'_t$ corresponding to (4) in Definition \ref{def:css} is of constant order (see for instance \cite[Lemma 5.7]{rouSe}).} Hence, the energy of $\Es_t$ in the domain $[0, \rpp] \ti \D$ is of order $O(\rpp) = O(\e R)$, 
		\be
			\int_{[0, x_-]\ti D}|\Es_t(z)|^2 \d z = O(\e R). 
		\ee
	
\vspace{3mm}\noindent	{\it Energy in ${[x_-, x_0]} \ti D$.}		 For the layer $[x_-, x_0]\ti D$, remember 
		$E_t^\scr = E_t^{(1)} + E_t^{(2)}$. Using the definition of $E_t^{(1)}$ in \eqref{def:E1}, we compute
		\be
			\int_{[x_-, x_0]\ti D} |E_t^{(1)}(z)|^2 \d z = \frac13 (x_0 - x_-) \int_D \bigl( \wt E'_t(x_0, y)\cdot \vec e_x\bigr)^2 \d y
		\ee
		and note by definition in \eqref{eq:xminusdef} we have 
			$x_0 - x_- \le 1 + 2 |d_0|.$ 
		Next, the field $E^{(2)}$ consists of $a_+ \le 3 |d_0| + 1$ 
		constituents $\{E^{(2, i)}\}_{i \le a_+-1}$. Since the length of the interval $\Ii$ is of constant order, we conclude from \cite[Lemma 5.7]{rouSe} that also the energy of $E^{(2, i)}$ is of constant order. Hence, 
		\be
			\int_{[x_-, x_0]\ti D} |E_t^{(2)}(z)|^2 \d z = O(a_+) = O(|d_0|).
		\ee
		Using $|E_t^{(1)}+ E_t^{(2)}|^2 \le 2 |E_t^{(1)}|^2 + 2|E_t^{(2)}|^2$ we deduce 
		\be
			\int_{[x_-, x_0]\ti D} |\Es_t(z)|^2 \d z \le O(|d_0|) \int_D \bigl( \wt E'_t(x_0, y)\cdot \vec e_x\bigr)^2 \d y + O(|d_0|).
		\ee

\vspace{3mm}\noindent		{\it Energy in $I_0 \ti D$.}
		Remember the configuration $\t\CC\in \Neut(\La)$ from \eqref{eq:bcchoice} and the fields $\wt E_t$ and $\wt E'_t$ from \eqref{def:Echoice} and \eqref{def:E'}. 
		By definition of $\Es_t$, 
		\be
			\int_{I_0\ti D} |\Es_t(z)|^2 \d (z, t) = \int_{I_0\ti D} |\wt E'_t(z)|^2 \d (z, t).
		\ee
Altogether, we find that
		\begin{multline} \label{eq:detscreen}
			\int_{\La_R} |\Es_t(z)|^2 \d z
				\le 	\int_{I_0\ti D} |\wt E'_t(z)|^2 \d z + O(|d_0|) + O(\e R) \\
					+ O(|d_0|) \int_D \bigl( |\wt E'_t(x_0, y)|^2 + |\wt E'_t(R -x_0, y)|^2 \bigr) \d y. 
		\end{multline} 

	Next, we compare $\wt E_t$ and $\wt E'_t$ in $[x_0, R - x_0]\ti D$. Remember $\wt E_t, \wt E'_t$ are gradient fields created by configurations $\t\CC, \tilde \CC'\in \Neut(\La)$, with  $\tilde \CC'$ obtained from $\t\CC$ in the regularization step \eqref{eq:regstep}. 
	
	 Lemma \ref{diff1Lem} bounds the change in electric field modulus caused by moving a single bridge in a given configuration. Since $x_0$ is $\e$-tame, one may apply this lemma at most {$4\e^{-4}$} successive times to obtain for $x \in I_0$
	\be\label{eq:using difflemm}
		|\wt E'_t(x, y) - \wt E_t(x, y)| \le 4\e^{-4}\cm\exp(-\cm' (x - x_0) \wedge (R - x_0 - x)).
	\ee
	Combining $a^2 = b^2 + (a - b)^2 + 2b (a - b)$ with the Cauchy-Schwarz inequality, we get 
	\begin{align*}
		\int_\D |\wt E'_t(x, y)|^2 \d y \le& \int_\D |\wt E_t(x, y)|^2\d y + \int_\D |\wt E'_t(x, y) - \wt E_t(x, y)|^2 \d y \\
		& + 2 \Big(\int_\D |\wt E_t(x, y)|^2 \d y \ti \int_\D |\wt E'_t(x, y) - \wt E_t(x, y)|^2 \d y \Big)^{1/2}. \nn
	\end{align*}
	Applying \eqref{eq:using difflemm} to the second and third terms on the right-hand side, we find that there exist $c_0(\e)>1$ and $c_0'(\e) > 0$ such that for all $x \in I_0$, 
	\begin{align}
		\label{enIntEq}
		\int_\D |\wt E'_t(x, y)|^2 \d y \le \int_\D |\wt E_t(x, y)|^2 \d y+ c_0\exp(-c_0' (x - x_0) \wedge (R - x_0 - x))\Big(1 + \int_\D |\wt E_t(x, y)|^2 \d y\Big).
	\end{align}
	In particular, 
	\begin{align}
		\label{enIntEq2}
		\int_\D |\wt E'_t(x, y)|^2 \d y \le 2c_0\Big(1 + \int_\D |\wt E_t(x, y)|^2 \d y\Big).
	\end{align}

			We return to \eqref{enIntEq}. For the first term on the right-hand side of \eqref{eq:detscreen}, we further subdivide $[x_0, R -x_0]\ti D$. Integrating \eqref{enIntEq2} over $$J_{x_0, R} = [x_0, x_0 + \sqrt R] \cup [R - x_0 - \sqrt R, R - x_0]$$ shows that for all large $R$, 
	\be \label{eq:detscreen2} 
		\int_{J_{x_0, R}\ti \D} |\Es_t(z)|^2 \d z \le 2c_0\Big(\sqrt R + \int_{[x_0, x_0 + \sqrt R] \ti \D}|\wt E_t(z)|^2 \d z\Big).
\ee
	Integrating \eqref{enIntEq} over all $x \in [x_0 + \sqrt R, R - x_0 - \sqrt R]$ shows that for all large $R$, 
	\begin{multline} \label{eq:detscreen3}
		\int_{[x_0+ \sqrt R, R - x_0 - \sqrt R]\ti \D} \hspace{-.2cm}|\Es_t(z)|^2 \d z 
 	\le \int_{\La_R} |\wt E_t(z)|^2 \d z 
			+ c_0\mathrm e^{- \sqrt R}\Bigl(1 +\int_{\La_R}|\wt E_t(z)|^2 \d z \Bigr).
	\end{multline} 
	Combining the inequalities \eqref{eq:detscreen}, \eqref{eq:detscreen2} and \eqref{eq:detscreen3}, we get
	\begin{align*} 
		\int_{\La_R} |\Es_t(z)|^2 \d z
				&\le 	\int_{\La_R} |\wt E_t(z)|^2 \d z + c_0 \mathrm e^{- \sqrt R} \Bigl( 1 + \int_{\La_R} |\wt E_t(z)|^2 \d z \Bigr) + O(|d_0|) + O(\e R) \\
				&\phantom\le+ 2 c_0 \Bigl( 2 \sqrt R + \int_{J_{x_0, R}}|\wt E_t(z)|^2 \d z \Bigr) \\
					&\phantom\le+ O(|d_0|) c_0 \int_{\pa I_0\ti D} |\wt E_t(z)|^2  \d z. 
	\end{align*} 
	To conclude, let $\vp_0$ denote the right-hand side. Then, we integrate $t$ over $[0, \b]$, divide by $2\b$, and use the bounds from Lemma \ref{goodBoundLem}. This yields 
	\begin{multline*} 
		\frac1{2\b} \int_0^\b \vp_0(\wt E_t, \e, R) \d t
				\le \frac1{2\b} \int_{\La_R\ti [0, \b]} |\wt E_t(z)|^2 \d(z, t) + c_0 \mathrm e^{- \sqrt R} \big( 1 + MR \big) + O(|d_0|) \\
				+ O(\e R)	+ 2 c_0 \big( 2 \sqrt R + {\frac M\e\sqrt R} \big) 
					+ O(|d_0|) c_0 \frac M\e
	\end{multline*} 
	Remembering $|d_0|< \frac{\e R}2$ from \eqref{eq:d0bound}, we obtain 
	\be
\frac1{2\b}		 \int_{[0, \b]}\vp_0(\wt E_t, \e, R) \d t
				\le \frac1{2R \b} \int_{\La_R \ti [0, \b]} |\wt E_t(z)|^2 \d (z, t) 
				+ \vp(\e, R)M
	\ee
	with 
	\begin{equation*}
		\limsup_{R\to \ff}\vp(\e, R) = O(\e). \qedhere
	\end{equation*}
\end{proof} 

\section{Quasi-continuity of the specific energy} \label{sec:quasicontinuity} 

In this section, we prove the quasi-continuity result, Proposition \ref{enBoundProp}.

\subsection{Proof strategy} 
The proof is a variant of arguments by {Lebl{\'e}} and Serfaty {\cite{serfInv}}. Let us sketch the heuristic argument before we embark on the details. The principal obstruction to Proposition \ref{enBoundProp} is the lack of a simple analog to Lemma \ref{semContLem} 
for an upper bound. In other words, it is difficult to show that $\Emp(\CC) \approx P$ implies a bound of the type $\frac1N H_{\La_N}(\CC) \lesssim \W(P)$. Screening is the mechanism by which we will achieve this. {In particular, we map a significant number of the configurations $\CC \in \Conf(\La_N)$ such that $\Emp(\CC)\approx P$, to other configurations $\Cs\in \Neut(\La_N)$ that have a similar empirical field}
\be
	\Emp(\Cs) \approx \Emp(\CC) \approx P
\ee
and with energy 
\be
	\frac1NH_{\La_N}(\Cs) \lesssim \W(P). 
\ee
The inequality holds true even though $\frac1N H_{\La_N}(\CC)$ might be much larger than $\W(P)$, thus as far as the energy is concerned, the configuration $\Cs$ is better-behaved than $\CC$.
We will find a probability measure $\Pc_N$ on $\Conf(\La_N)\ti \Neut(\La_N)$ such that 
\begin{itemize} 
	\item [(i)] {The first marginal is the Poissonian reference measure projected to $\La_N$, i.e., the distribution $\Pois_{\La_N}$ where $\Pois$ is a Poisson process with independent Brownian bridge marks defined above (\ref{def:entP}). The second marginal corresponds to a Binomial process with independent Brownian bridge marks as {defined around} (\ref{def:BBmeasure}).} We shall write elements of $\Conf(\La_N)\ti \Neut(\La_N)$ as $(\CP, \Cco)$. 
	\item [(ii)] With high probability, the coupled Binomial process $\Cco$ is close to the screened configuration $\Cs = (\CP)^{\ms{scr}}$ in the sense that 
	\be
			\frac1N H_{\La_N}(\Cco) \lesssim \frac1N H_{\La_N}(\Cs), \qquad \Emp(\Cco) \approx \Emp(\Cs).
	\ee
\end{itemize} 
Then we may estimate for some $c>0$ and any $\delta>0$,
\begin{align*}
	\P_N\Bigl( \Emp(\CC) \approx P, \, \frac1NH_{\La_N}(\CC) \lesssim \W(P)\Bigr)
	 & =\Pc_N \Bigl( \Emp(\Cco) \approx P, \, \frac1NH_{\La_N}(\Cco) \lesssim \W(P)\Bigr) \\
	 & \gtrsim \Pc_N \Bigl( \Emp(\Cs) \approx P, \, \frac1NH_{\La_N}(\Cs) \lesssim \W(P)\Bigr) \\
	 & {{\gtrsim e^{-c\delta N}\Pc_N \Bigl( \Emp(\CP) \approx P\Bigr)}}\\
	 &\approx \exp \bigl( - N (\ent (P)+c\delta)\bigr),
\end{align*}
which is the kind of statement that we are after. 

For the rigorous proof, it is convenient to approximate empirical fields by block averages. Consider volumes 
\be \label{eq:blockunion} 
	\La_N = \bigcup_{i \le m} K_i, \quad K_i = [(i-1) R, iR)\ti D, \quad N = mR.
\ee
The block average of a configuration $\CC \in \ms{Conf}(\La_N)$ is 
\be\label{def:block average}
\DA(\CC) := \frac1m \sum_{i \le m} {\de_{\th_{x_i}}(\CC_{K_i}), \quad x_i = (i-1)R}.
\ee
The $m$- or $N$-dependence is suppressed from the notation. Notice that 
 $\DA(\CC)$ puts full mass on $\Conf(\La_R)$.
We shall see {in Step 3 of the proof of Proposition \ref{enBoundProp}, } that if the block average is close to a given shift-invariant measure $P\in \Ps$, then so is the empirical field.
{In our topology,} the block average is not sensitive to small changes of the configuration in the vicinity of block boundaries $\pa K_i$; this allows us to apply the screening construction from Section \ref{sec:screening}, block by block without changing too much the block and empirical averages. 

The coupling measure is defined blockwise.

{\subsection{Coupling within block $K_1$} \label{sec:coupling}
Let $(\Omega, \mc F, \P)$ be an enlarged probability space big enough to support all random variables introduced below.
We now describe the coupling measure in the fixed block $K_1$, which will eventually be used to form $\Pc_N$ on the full domain $\La_N$ with the correct marginal distributions as described in the beginning of the section. Throughout this subsection, we often suppress the dependence on $K_1$, by writing for example $(\CP, \Cco)=(\CP_{K_1}, \Cco_{K_1})$.
The joint distribution of $(\CP, \Cco)$ under the coupling measure $\Pc$ on $K_1$ will be defined using (regular) conditional probabilities. In particular, the first marginal $\CP$ simply has the marginal distribution $\Pois_{K_1}$, and we proceed to define the conditional distribution of the second marginal given $\CP$.}

As a first step, we let the coupling measure be a mixture 
$$\Pc = \frac12 \Pca + \frac12 \Pcb,$$
where under $\Pca$, the marginals are independent, i.e., the configuration $\Cco$ is independent from $\CP$ and has a distribution given by a 
 Binomial process with independent Brownian bridge marks as {defined around}  (\ref{def:BBmeasure}). The independent portion $\Pca$ is used later in Lemma \ref{fcProbLem} to handle configurations which are not $(M,\e)$-regular.

Under $\Pcb$ the marginals are dependent, and we construct the coupling as follows. 
Fix \mbox{$\de\in(0,1/2)$} {(before any choice of $M,\e,R$)} and put
{\be
	\La_{R, \e}:= [6\e R, R - 6\e R]\ti D.
	\ee}
 It is convenient to label some of the bridges in the configuration $\CP$. Suppose $\RP$, $\{\ZP_i\}_{i \ge 1}$, $\{\BP_i\}_{i \ge 1}$ are independent random variables such that: 
\begin{itemize}
	\item $\RP$ is a Poisson random variable with parameter $R - {12 \e R}$, 
	\item $\ZP_i$ is uniformly distributed in {$\La_{R, \e}$},
	\item $\BP_i$ is a Brownian bridge with $\BP_i(0) = 0$ and diffusivity 1,
\end{itemize} 
and such that $\CP_{{\La_{R, \e}}}:\Omega\to \Conf$ \tcb{can be represented as}
\be
 \bigcup_{i \le \RP } \{(\ZP_i, \ZP_i+ \BP_i)\}.
\ee
 We now specify the conditional distribution of $\Cco$ given $\CP$ by \tcb{fixing $B_i^\mathsf B := B_i^\mathsf P$ for all $i \le R$ and by} defining the initial points of the configuration $(\ZB_1, \dots, \ZB_R)$ as follows. Set 
{\be
	\de_1 := (\de - 12\e R) / (1 - 12\e R).
\ee}
Then, \tcb{treating each $i \le R$ independently}, with probability $1 - \de$ the variable $\ZB_i$ is \tcb{set to be} equal to $\ZP_i$, with probability $\de_1$ it is uniformly distributed in $\La_ R$, and with probability $\de - \de_1$ it is uniformly distributed in $\La_R\setminus {\La_{R, \e}}$.

If $\RP \ge R$ this determines the vector $(\ZB_1, \dots, \ZB_R)$---where in this case the points $\ZP_{R + 1}, \dots, \ZP_{\RP}$ are simply discarded. 

If $\RP < R$, we complete the definition by letting $\ZB_i$ be uniformly distributed in $\La_R$, for all $\RP < i \le R$ (we also produce extra independent random variables $\BP_i$ for all $\RP < i \le R$).

A technical description of the coupling under $\Pcb$, when $\CP$ is $(M,\e)$-regular, is via the probability kernel
\be
	K_{\ms{reg}}(z, \d z'):= (1 - \de) \de_z (\d z') + \de_1\frac1 R\, {\one_{\La_{R, \e}}(z') \d z' + (\de - \de_1)\frac1{12 \e R}\, \one_{\La_R \setminus \La_{R, \e}}(z') \d z'.}
\ee
Then, {$\P$-almost surely} on the event $\RP \le R$, for all measurable $A_1, \dots, A_R\su S$, 
\be\label{cpkernel}
	\P\big( \ZB_1\in A_1, \dots, \ZB_R \in A_R\big| \RP, \ZP_1, \dots, \ZP_{\RP} \big)
	 = \Big( \prod_{i \le \RP}K_{\ms{reg}}(\ZP_i, A_i) \Big) \ti \prod_{\RP+1 \le i \le R} \frac{|A_i|} R.
\ee
Finally, we put 
\be
	{\Cco} = \bigcup_{i \le R}\{(\ZB_i, \ZB_i+\BP_i)\},
\ee
and set $\Pcb$ on $\Conf(\La_R)\cap {\{(M,\e)\text{-regular}\}}\ti \Neut(\La_R)$ to be the joint law of $(\CP, \Cco)$. By construction, $\Pc$ is indeed a coupling. That is, it has the correct marginals.

{After construction of a measure with the correct marginals, the big space $\Omega$ is no longer of any use. We change notation slightly and from now on use the letter $\Pc$ for the distribution of $(\omega^\mathsf P, \omega^\mathsf B)$.}

\begin{lemma}\label{marg correct}
The coupling measure $\Pc$ on $\Conf(\La_R)\ti \Neut(\La_R)$ has as its first marginal
 distribution $\Pois_{\La_R}$ as in \eqref{def:entP}, and has as its second marginal distribution a Binomial bridge process as {defined around} \eqref{def:BBmeasure}.
\end{lemma}

Next, we define the event of coupling to a screened configuration $\Fc\su \Conf(\La_R)\ti \Neut(\La_R)$ and bound $\Pc(\Fc | \CP)$ from below. Loosely speaking, $\Fc$ imposes that $\Cco$ is at most a small perturbation of {the screened configuration} $\Css$ defined in Section \ref{sec:screening}. {As mentioned above, here the screening construction is applied to $\CP$ and $\Css = (\CP)^\mathsf{scr}$.}

\medskip \noindent 

\begin{definition}[{Coupling to screened configurations}]
	The event $\Fc$ is defined by setting $A_{\ms{reg}} = \{\CP \text{ is }(M, \e)\text{-regular}\}$ and letting 
\begin{align}
		\Fc \cap A_{\ms{reg}}^c := \big\{\inf_{b' \in \Cco }|b - b'|_\ff \le 1/16 \text{ for every $b \in \Css$}\big\}\cap A_{\ms{reg}}^c
\end{align}
and 
$$
	\Fc \cap A_{\ms{reg}} := {\big\{\Cco_{\La_+} = \Css_{\La_+}\big\}} \cap \big\{\inf_{b' \in \Cco}|b - b'|_\ff \le 1/16 \text{ for every $b \in \Css \setminus \Css_{\La_+}$}\big\}\cap A_{\ms{reg}}.
$$
\end{definition}

The rest of this subsection is devoted to establishing a lower bound on the conditional probability of $\Fc$.

\begin{lemma}[{Probability of screening under the coupling}]
	\label{fcProbLem}
	There exists $c = c(P)$ with the following property. For all $\de \in (0, 1/2)$, $M > 2$ there exist $ \e_0(\de) > 0$, $R_0 = R_0(\e) > 0$ such that 
	\begin{align*} \Pc[\Fc\, |\, \CP]
	\ge \exp(-c R) + \exp(-c \de R )\one\{\CP \text{ is $(M, \e)$-regular}\}.\end{align*}
	holds for all $\e < \e_0(\de)$ and $R \ge R_0(\e)$.
\end{lemma}

\begin{proof}
	\emph{\noindent Case 1: $\CP$ is not $(M, \e)$-regular.} 
	
	Note that in this case $\Css$ is just $R$ straight lines at locations \eqref{screeningirreg}.
	Under $\Pca$, the probability that the starting point $b(0)$ of a fixed bridge $b \in \Cco$ is at distance at most $1/32$ from the starting point $b'(0)$ of a fixed bridge $b' \in \Css$ equals $1/(16R)$. Moreover, let $q_0 > 0$ be the probability that a Brownian bridge does not deviate by more than $1/32$ from its starting point. Then, the probability for that fixed bridge $b$ to satisfy $|b - b'|_\ff \le 1/16$ is at least $q_0/(16R)$.
	Since under $\Pca$, the coupled process $\Cco$ consists of $R$ i.i.d.\ bridges independent of $\CP$, we conclude that for all large $R$, via Stirling's approximation, that
	\begin{align} \label{eq:x0prime}
			\Pc(\Fc | \CP) \ge \frac12{R!\Big(\frac {q_0}{16R}\Big)^R \ge \frac 1 2\Big(\frac{q_0}{16e}\Big)^R.}
	\end{align}
	\emph{\noindent Case 2: $\CP$ is $(M, \e)$-regular.}
				To begin with, we assert that $\RP\le R$. Indeed, choosing a good boundary $x_0 \in [2\e R, 6\e R]$ using Lemma \ref{goodBoundLem}, we derive a bound on the charge imbalance {(=net charge)} in the domain $I_0 \ti D = [x_0, R - x_0] \ti \D$ at time 0 as in the proof of Lemma 	\ref{noFracLem}. More precisely, by part (2) of Lemma \ref{lem:chargimb1}, this charge imbalance is at most
	\be
	\label{eq:chargeImb3}
	 \Bigl({\sum_{x'\in \partial I_0} \frac4\b\int_{\D\ti [0, \b]} |\wt E_t(x',y)|^2 \d (y, t)} \Bigr)^{1/2}
	+ \sum_{x'\in \pa I_0}\#\big\{b\in \tilde \CC:\, \inf_{t \le \b}\ |b_x(t) - x'| \le 1 \big\}.
	\ee
	{(Here $\tilde \CC = \tilde{\omega^\mathsf P}$ is defined similarly to~\eqref{eq:bcchoice}.)} 
	The cardinalities on the right are bounded by {$4\e^{-4}$} because $x_0$ and $R-x_0$ are $\e$-tame.      Hence, we conclude as in the proof of Lemma \ref{noFracLem} that \eqref{eq:chargeImb3} is at most  $9\e^{-4}$. On the other hand, if $\RP \ge R$, then the charge imbalance at time $0$ is at least $ \RP- (R - 4\e R) \ge 4\e R$. Thus, $4\e R \le 9\e^{-4}$, which is absurd once $R > 3/\e^5$.

		Next, we work with $\Pcb$ and consider the domains $\La_+$ and $\La_R \setminus \La_+$ separately. By the construction of the coupling, when considering only the first $\RP$ bridges of the coupling, i.e. those produced by $K(z,\d z')$ in \eqref{cpkernel}, we have that $\Cco_{\La_+} = \Css_{\La_+}$ has a probability of at least \begin{align}
		\frac 1 2(1 - \de)^{\RP} \ge \frac 1 2(1 - \de)^R.
		\end{align}
	Now, leveraging similar charge-imbalance arguments as in \eqref{eq:chargeImb3}, we conclude that 
	\begin{align}
		\#\Css_{\La_+} \ge R - 20\e R - 9\e^{-4}.
	\end{align}
	At most $a = 21\e R$ bridges remain to be coupled. Arguing as in the irregular case, the probability that the remaining bridges from $\Cco$ are close to the remaining bridges of $\Css$ has probability at least 
	$$ \frac 12\frac{a! (\de_1 q_0)^a}{(16R)^a} \ge \frac 12\Big(\frac{a \de_1 q_0}{16e R}\Big)^a = \frac12\exp\big(21\e R \log(21\e \de_1q_0/(16e))\big), $$
	which is of order $\exp(-o(\e)R)$.
\end{proof}

{\subsection{Screened configurations have compatible electric fields} 
Continue to assume that the coupling measure is for the fixed block $K_1$.} We have the following extension of Proposition \ref{prop:detscreening}.

\begin{lemma}[Screened fields under the coupling]
	\label{fcScreenLem}
	Let $M > 2$, $\e \in (0, 1/ M^2)$ and $R > (2/\e)^8$.
	There exists an $M_0>1$ not depending on $M,R,\e$, and moreover, on the event $\Fc$, there exists a field $\{\EB_t\}_{t\le\b}$ that is screened and compatible with $\Cco$ in the sense of \eqref{eq:compatible} and \eqref{eq:screened}, such that
	\begin{enumerate}
		\item if $\CP$ is not $(M, \e)$-regular, then for every $t \le \b$
		$$\int_{\La_R} \bigl|\EB_t(z)\bigr|^2 \d z \le \vp_0(\wt E_t, \e, R),$$
			and 
		\be
			\frac1{2\b}\int_{[0, \b]}\vp_0(\wt E_t, \e, R) \d t \le M_0 R,
		\ee

		\item if $\CP$ is $(M, \e)$-regular, then  for every $t\le \b$
		$$\int_{\La_R} \bigl|\EB_t(z)\bigr|^2 \d z \le \vp_0(\wt E_t, \e, R),$$
			and 
		\begin{align}
			\label{screenEq}
				\frac1{2\b}\int_{[0, \b]} \vp_0(\wt E_t, \e, R) \d  t \le \tH_{\La_R}(\CP) + \vp(\e, R) MR,
		\end{align}

	\end{enumerate} 	
			where $\vp_0(\wt E_t, \e, R) \ge 0$ {is measurable in $t$} and  $\vp(\e, R) \ge 0$  satisfies
																	                        \be 
																				                                        \limsup_{\e\to 0} \limsup_{R\to \ff}\vp(\e, R) = 0.
																									                        \ee

\end{lemma}

\begin{proof}
We give a sketch of proof. For the existence of $\EB_t$ one revisits the proof of Proposition~
	\ref{prop:detscreening} and checks the robustness of the construction with respect to the perturbations of $\Cco$ from $\Css$.
We note that in the irregular case, the configuration $\Css$ consists simply of evenly-spaced straight lines (see \eqref{screeningirreg}).
	{For the case where $\CC$ is $(M, \e)$-regular,  	 the coupled bridges $\{b_i\}$ of $\Cco$ stay away from the boundaries of the respective cells $\Ii$ or $[i - 1, i]\ti D$ and also away from the boundary $x_0$ since $\eta + 1/16 < 1/3$, see Definition \ref{def:css} and \eqref{2d_0 choice}.}
	The energy bounds in \cite[Lemma 5.7]{rouSe} are robust with respect to the perturbations of the smeared $\de$-point charges as long as they stay inside the integration domain. Hence, the arguments for the energy of $\EB_t$ in the domain $[0, \rmm]$ and for the energy of $E^{(2)}$ in the domain $[\rpp, x_0]$ carry over to the perturbed setting. Moreover, the electric field $E^{(1)}$ is independent of the perturbation. Finally, when considering the energy in the bulk $\La_+$, we note that the bounds from Lemma \ref{diff1Lem} also do not depend on the precise location of the bridges.
\end{proof}

\subsection{Most configurations are regular} 
\label{regSec}

We now show that under bridge processes of finite energy and entropy, it is highly likely to see $(M, \e)$-regular configurations.

\begin{lemma}[Most configurations are regular]
	\label{expAppCor}
	Let $P\in \Ps$ be such that $\ent(P) \vee \W(P) < \ff$. There exists $c = c(P) > 0$ with the following property.
	If $M > 2$ and $\e \in (0, 1/M^2)$, then 
	{\begin{align}
		\label{expAppBound}
		\liminf_{R \to \ff}		P\bigl( \CC_{\La_R}\text{ is $(M, \e)$-regular}\bigr) \ge 1 - \frac{\W(P)} M - c\e^{1/3}.
	\end{align}}
\end{lemma}
\noindent 
It follows in particular that if $\W(P)$ and $\ent(P)$ are both finite, then
\be
	\lim_{M \to \ff}\lim_{\e \to 0}\liminf_{R \to \ff} \ P\bigl( \CC_{\La_R}\text{ is $(M, \e)$-regular}\bigr) = 1. 
\ee
\begin{proof}
	 {By the Markov inequality and Proposition \ref{superAddLem}, 
		\begin{align*}\limsup_{R \to \ff} MP\Big(\WR(\CC) \ge M\Big)
			\le \limsup_{R\to \ff}\E_P\big[(R^{-1}\tH_{\b, \La_R}(\CC))\wedge M\big]
				&= \W^M(P).
		\end{align*}
		For the excursion condition of $(M,\e)$-regularity, we apply Lemma \ref{lem:entropy-estim} with $\z = \e^{-7/3}$ to get 
				\begin{align*}
					P\Big(\sum_{b \in \CC}(\psi(b)^{7/6} - \e^{-7/3})_+\ge \e R/2\Big) &\le \frac2{\e R}\E_P\Big[\sum_{b \in \CC}\big(\psi(b)^{7/6} - \e^{-7/3}\big)_+\Big]
\le c\e^{11/24}
		\end{align*}
		where $c$ may depend on $\ent(P)$.}
\end{proof}

\subsection{Proof of Proposition \ref{enBoundProp}} 
Fix $P \in \Ps$ with $\ent(P) \vee \W(P) < \ff$ and a neighborhood $U_{F, P}$ of the form (\ref{topology-basis}) where we may assume that the local test functions in $F = \{f_1, \dots, f_n\}$ satisfy
$\max_k\|f_k\|_\ff\le1$. Since each of the test functions is local, we also denote $d$ to be the maximum diameter of dependence over all $f_k$ in $F$. That is, $f_k(\CC) = f_k(\CC_\Delta)$ where $$\Delta := [- d/2, d/2]\ti D\su S.$$

		\medskip
	\noindent
	\textbf{Step 1: Concatenate screened blocks.}
	The coupling definition from Section \ref{sec:coupling} and Lemma \ref{fcScreenLem} extend from the block $\ol{K_1} = [0, R]\ti D$ to other blocks $\ol{K_i}$, with the obvious modifications. 
This provides us with measures $\Pc_i$ on $\Conf(K_i)\ti \Neut(K_i)$, events $$\Fc_i\su \Conf(K_i)\ti \Neut(K_i),$$ and fields $\EB_{i, t}: \ol{K_i} \to \R^{k + 1}$ satisfying among other conditions, 
\be
	\EB_{i, t} ((i - 1) R, y) \cdot \vec e_x = \EB_{i, t}(i R, y) \cdot \vec e_x
\ee
for all $t \le \b$, $i \le m$ and $y\in D$. 

	Writing $\{(\CP_i, \Cco_i)\}_{i \le m}$ for a collection of independent configurations distributed according to the product measure $\bigotimes_{i \le m} \Pc_i$, the superposition $\cup_{i \le m}\Cco_i$ is distributed according to a Binomial bridge process on $\La_N$ {conditioned on the event 
		$$A_{N, R}:=\bigcap_{i \le m}\{\#\Cco_{K_i} = R \}$$
		 of having exactly $R$ bridges in each block $K_i$. 	Hence, we define the global coupling $\Pc$ for $(\CP, \Cco)$ on $\La_N$ as a mixture that has weight $\P(A_{N, R})$ on the product measure $\bigotimes_{i \le m} \Pc_i$ and weight $1-\P(A_{N, R})$ on an independent coupling between a Poisson bridge process on $\La_N$ and a Binomial bridge process on $\La_N$ conditioned to have at least one block not containing exactly $R$ bridges. We then set
		 $$\Fc_{\La_N}:= \bigcap_{i \le m}\big\{(\CP_{K_i}, \Cco_{K_i}) \in \Fc_i\big\}.$$
Since, on the event $\Fc_{\La_N}$, the electric fields are screened for every $i \le m$ and $t \le \b$, their superposition defines a globally screened electric field 
\be
	\EB_t(z) = \sum_{i \le m} \one_{K_i}(z) \EB_{i, t}(z)
\ee
without creating additional divergence.
 Hence, on $\Fc_{\La_N}$
\be
	\div \EB_t = - \sum_{b\in \Cco} \de_{b(t)}^\eta + \one_{\La_N}.
\ee

	\medskip
\noindent 
\textbf{Step 2: Compare energies.} 

Fix $\de < 1$, which we may assume to satisfy $9(M_0 \W(P) + 3 + d)\de < 1$ and set $M  = 4/ \de^2$.

 Also let $\e <1/M^2$ be such that $\lim_{R \to \ff}\vp(\e, R) < \de^2/M$, such that Lemma \ref{fcProbLem} applies, and such that the error $c\e^{1/3}$ from Lemma \ref{expAppCor} is smaller than $\de^2$.
 
Finally, choose {$R > (2/\e)^8 \vee 18d/\de$} large enough that $\WRM(P) \le \W(P) + \de ^2$, such that Lemmas \ref{fcProbLem} and \ref{fcScreenLem} hold, and such that the error in Lemma \ref{expAppCor} is smaller than $\de^2$.

By \cite[Lemma 3.10]{serfInv}, replacing a screened field compatible with a charge-neutral configuration with the gradient of the potential associated to that configuration (which may possibly have different boundary conditions) can only decrease the energy. In other words, for every $t \le \b$,
\be \label{going to H}
 \int_{\La_N} \bigl|\nabla V_t(z, \Cco, \La_N)|^2 \d z \le \int_{\La_N} |\EB_t(z)|^2 \d z.
\ee
Let 
\be
	\II(\CC) := \big\{i \le m:\, \CC_{K_i} \text{ is $(M, \e)$-irregular}\big\}
\ee
be the set of indices corresponding to blocks with $(M, \e)$-irregular configurations.
By \eqref{going to H} and Lemma \ref{fcScreenLem}, on the event $\Fc_{\La_N}$, 
\be\label{128}
\begin{aligned}
	\frac1N H_{\La_N}(\Cco) & {\le \frac{\#\II(\CP) M_0 } m + \frac1m \sum_{i=1}^m \min\Bigl( \frac1 R \tilde H_{K_i}(\omega^\mathsf P),M\Bigr)+ \vp(\e, R) M}\\ 
		&= \frac{\#\II(\CP) M_0 } m + \WRM(\DA(\CP))+ \vp(\e, R) M.
\end{aligned}
\ee
Now, let 
\be \label{eq:neigh}
\begin{aligned}
	U_1(P) &:= \big\{Q\in \mc P:\, \bigl| \WRM(Q) - \WRM(P)\bigr|<\de^2\big\}\\
	U_2(P) &:= \big\{Q\in \mc P:\, \bigl| Q( \CC_{\La_R}\text{ is $(M, \e)$-irregular}) - P ( \CC_{\La_R}\text{ is $(M, \e)$-irregular}) \bigr|<\de^2\big\}.
\end{aligned} 
\ee
Then, if $\DA(\CP) \in U_2$, Lemma \ref{expAppCor} and the choices of $M$, $\e$ and $R$ give that 
\begin{align}
	\label{ibEq}
	\frac{\#\II(\CP)} m \le P( \CC_{\La_R}\text{ is $(M, \e)$-irregular}) + \de^ 2
	&\le {(\W(P) + 3)\de^2.}
\end{align}
Moreover, $\WRM(P) \le \W(P) + \de ^2$ and $\DA(\CP) \in U_1$ imply that 
\be
	\WRM(\DA(\CP)) \le \WRM(P) + \de^2 \le \W(P) + 2\de^2. 
\ee
The assumption on $\de$ from the beginning of this paragraph yields that on the event {$\Fc_{\La_N} \cap \{\DA(\CP) \in U_1 \cap U_2\}$,}
{\be \label{eq:energyfinal}
	\frac1N H_{\La_N}(\Cco) \le \W(P) + \big(M_0(\W(P)+3) + 3\big)\de^2 \le \W(P) + \de.
 \ee}
	\medskip
	\noindent 
\textbf{Step 3: Compare block and empirical averages.}  
We claim that, 
up to a null set of $\Pc$,
\be\label{u3 implication}
 \Fc_{\La_N}\cap \{\DA(\CP)\in U_3\}\su\{\Emp(\Cco) \in U_{F, P}\},
\ee
where 
\begin{align} \label{eq:neigh3}
	U_3 := \big\{Q\in \mc P:\, \max_{k \le r} \bigl| \E_Q[\bar f_k] - \E_P[\bar f_k]\bigr|< \de/3\big\}. 
\end{align}
with 
\be
	\bar f_k(\CC):= \frac1R \sum_{r \le R - 1} f_k(\th_r( \CC)).
\ee
Since $P$ is shift-invariant we have $\E_P[\bar f_k] = \E_P[f_k]$, so that
\begin{align} \label{eq:compare}
	\bigl|\E_{\Emp(\Cco)}[f_k]- \E_P[f_k]\bigr|
		&\nn\le \bigl| \E_{\Emp(\Cco)}[f_k] - \E_{\DA(\Cco)}[\bar f_k]\bigr| 
					+ \bigl|\E_{\DA(\Cco)}[\bar f_k] - \E_{\DA(\CP)}[\bar f_k]\bigr| \\
	&\phantom\le				+ \bigl|\E_{\DA(\CP)}[\bar f_k] - \E_P[\bar f_k]\bigr|. 
\end{align} 
It suffices to show that each of the three summands is bounded by $\de/3$. 

If $\DA(\CP) \in U_3$, then the third term on the right-hand side is smaller than $\de/3$. 

For the first term, observe 
\be
\begin{aligned}
	\E_{\Emp(\Cco)}[f_k] & = \frac1 m\sum_{j \le m} \bar f_k(\th_{x_j}(\Cco)), 
\end{aligned} 
\ee
where $x_j = (j-1)R$ is the left endpoint of $K_j$, 
and 
\be
\begin{aligned}
	\big| \E_{\Emp(\Cco)}[f_k]- \E_{\DA(\Cco)}[\bar f_k] \big| & \le \frac1m \sum_{j \le m} \bigl| \bar f_k(\th_{x_j}(\Cco)) - \bar f_k (\th_{x_j} (\Cco_{K_j})) \bigr|.
\end{aligned}
\ee
For $j = 1$ and $\CC \in \Conf$, we bound 
\be \label{eq:compare2}
\begin{aligned} 
	\bigl| \bar f_k(\CC) - \bar f_k(\CC_{K_1}) \bigr| & \le \frac1 R \sum_{r \le R - 1} \bigl| f_k(\th_r (\CC)) - f_k(\th_r(\CC_{K_1}))\bigr| \\
		&\le \frac1 R\, \#\big\{r \le R - 1:\, \pi_\Delta(\th_r (\CC)) \ne \pi_\Delta(\th_r(\CC_{K_1}))\big\}\\
		&\le 	\frac1R\#\big\{r \le R - 1:\, \th_{-r}\Delta \nsubseteq K_1 \big \},
\end{aligned}
\ee
which is at most $2d/R \le \de/3$.

Finally, for the second term on the right-hand side of \eqref{eq:compare}, proceeding as in \eqref{eq:compare2} yields
\begin{align*}
	\bigl|\E_{\DA(\Cco)}[\bar f_k] - \E_{\DA(\CP)}[\bar f_k]\bigr|&\le \frac1m \sum_{j \in \II(\CP)} \bigl| \bar f_k(\th_{x_j} (\Cco_{K_j})) - \bar f_k(\th_{x_j}( \CP_{K_j})) \bigr| \\
	&\phantom= +\frac1m \sum_{j \not \in \II(\CP)} \bigl| \bar f_k(\th_{x_j} (\Cco_{K_j})) - \bar f_k(\th_{x_j}( \CP_{K_j})) \bigr| \\
&\le \frac{2\#\II(\CP)}m + \frac{2d}R,
\end{align*}
where we have used that on the event $\Fc_{\La_N}\cap \{\CP \text{ is $(M,\e)$-regular}\}$, the configurations $\Cco_{K_1}$ and $\CP_{K_1}$ agree on $[6\e R, (1 - 6\e)R] \ti D$.
{The right-hand side above is at most $\de/3$ by \eqref{ibEq} and the choice of $R$.}

\medskip

\noindent \textbf{Step 4: Conclusion.} 
First, by Lemma \ref{marg correct}, \eqref{eq:energyfinal} and \eqref{u3 implication},
\begin{align*}
	&\P_N\Big( \Emp(\CC) \in U_{F, P}, \ N^{-1} H_{\La_N}(\CC) \le \W(P) + \de \Big)\\ 
	&\quad = \Pc \Big( \Emp(\Cco) \in U_{F, P}, \ N^{-1} H_{\La_N}(\Cco) \le \W(P) + \de \Big)\\
	&\quad\ge \Pc\big( \Fc_{\La_N}\cap \{\DA(\CP) \in U\}\big),
\end{align*}
where $U:= U_1\cap U_2 \cap U_3$ with $U_1, U_2, U_3$ given in \eqref{eq:neigh} and \eqref{eq:neigh3}.
Now, by Stirling's approximation, the probability
$$\P(A_{N, R})= \Pc({\#\Cco_{K_i}} = R \text{ for every $i \le m$})=\frac{N!}{(R!)^m}$$
of the Binomial point process consisting of $R$ bridges in each block $K_i$ is bounded below by $\exp(-\frac{cN\log R} R)$ for some $c > 0$. Hence, by Lemma \ref{fcProbLem} and \eqref{ibEq},
\begin{align*}
	\Pc\bigl[ \Fc_{\La_N}\, \big|\, \CP\bigr]\ge \P(A_{N, R})\prod_{i \le m}\Pc_i\big({(\CP_i , \Cco_i ) \in \Fc_i }\, \big|\, \CP\big) \ge \P(A_{N, R}) e^{-c \de N - c R\#\II}
	\ge e^{-c_1 \de N}
\end{align*}
for some $c_1>0$. Therefore,

\begin{align*}
	  \Ec\Bigl[ \Pc\bigl[ \Fc_{\La_N}\, \big|\, \CP\bigr] \one{\{\DA(\CP) \in U\}} \Bigr]
	& \ge \Ec\Bigl[ \exp( - c_1 \de N )\one{\{\DA(\CP) \in U\}} \Bigr]\\
	& = \exp( - c_1 \de N )\Pc \bigl( \DA(\CP) \in U\bigr).
\end{align*} 
Finally, by Sanov's theorem \cite[Theorem 6.2.10]{dz98},
	\begin{align}
	\label{sanovEq}
	\liminf_{N \to \ff}\frac1m \log\P\big(\DA(\CP) \in U\big) \ge - \inf_{Q\in U} \ent(Q\mid \Pois_{\La_R}) \ge -\ent(P_{\La_R}|\Pois_{\La_R}).
	\end{align}
	We conclude with $\ent( P) = \sup_R \ent_{\La_R}(P_{\La_R})${, see \cite[Remark 2.5]{georgii2}}.	
\qed

\subsection*{Acknowledgement.}
The authors thank the anonymous referees for providing us with high-quality feedback. Their constructive feedback helped us to improve the presentation in many places. The authors also thank D.~Chafa\"i, A.~van Enter and D.~Garc\'ia-Zelada for useful remarks on earlier versions of the manuscript.

\appendix
\section{{Energy in terms of electric field}}\label{sec:aux}

{Here, we prove Lemma \ref{splitLem} using a standard integration by parts.}

\begin{proof}[Proof of Lemma \ref{splitLem}] 
	In order to simplify notation, we fix $t$ and write $$V_N(z) = V(z, t, {\CC_{\La_N}}, \La_N).$$ Also for such fixed $t$, set $\{z_i\} = \{b_i(t):b_i\in{\CC_{\La_N}}, i\le N\}$. 
 Partial integration gives
	\begin{align*}
	\int_\T|\nabla V_N(z)|^2 \d z = \int_\T V_N(z) (-\Delta V_N(z)) \d z + \lim_{R \to \ff}\int_{\pa \La_R }V_N(z) \nabla V_N(z) \cdot \vec n \, \d z , 
	\end{align*}
	where $\vec n$ denotes the outer normal vector to the interface $\pa \La_R$. 	For each $\CC_{\La_N}$, there is an $R$ such that $\{z_1, \dots, z_N\}\su \La_R$, thus by charge neutrality, the second term on the right-hand side vanishes. 
		On the other hand, 
	\begin{align}\label{eq:part int}
	\int_{\T }V_N(z) (-\Delta V_N(z)) \d z = \sum_{i \le N}\frac1{|\pa B_\eta(o)|} \int_\T V_N(z) \de_{z_i}^\eta(\d z) - \int_{\La_N} V_N(z) \d z.
	\end{align}
	Now, insert the integral representation of the potential $V_N$ via the Green's function, i.e., 
	$$V_N(z) = \int_\T \gN(z - z') \Big(\sum_{i \le N} \de^\eta_{z_i} - \one_{\La_N}\Big)(\d z').$$
	In particular, the sum on the right-hand side of \eqref{eq:part int} becomes
	$$\sum_{i, j \le N}\gN^{\eta, \eta}(z_i - z_j) - \sum_{i \le N}\int_{\pa B_\eta(z_i)}\int_{\La_N} \gN(z - z')\d z' \de_{z_i}^\eta(\d z), $$
	whereas the subtracted integral on the right-hand side of \eqref{eq:part int} is
	$$\sum_{i \le N} \int_S\int_{\La_N} \gN(z - z')\d z\de_{z_i}^\eta(\d z') - \int_{\La_N} \int_{\La_N} \gN(z - z')\d z\d z'.$$
	Hence, splitting of the diagonal contributions in the double sum over $i$ and $j$ and re-arranging terms concludes the proof.
\end{proof}

\section{{Notation Guide}}

\begin{table}[h!]
	\begin{center}
		\caption{Notation Guide}
		\label{tab:table1}
		\begin{tabular}{r|l|l}
			Symbol & Object & Where defined \\
			\hline
			$\La_N$& $[0, N] \ti \D$ where $D = [0, 1]^k$&\eqref{def:TN} \\
			$	\H_N^\eta$ & classical potential energy of a configuration with smeared pts. & \eqref{potEnEetaEq}\\
			$\pi_\La(\CC)$ & projection of configuration $\CC$ onto $\La$& \eqref{def:proj}\\
			$P_\La$ & image of $P$ under $\pi_\La$& \eqref{projdef}\\
						$\Conf(\La)$ & space of bridge configurations $\CC$ in $\La$ &\eqref{def:conf}\\
						{$\P_N(\d\CC)$, $\E_N(\cdot)$}& bridge configuration measure on $\Neut(\La_N)$ & \eqref{def:BBmeasure} \\	
			{$ \mun(\d\CC)$}& Gibbs measure for bridge configurations &\eqref{def:muNb}\\

			$\Emp(\CC)$& empir. field constructed from $\CC$; a prob. meas. on config. space &\eqref{def:EMP}\\
						$\Pu$ &space of probability measures on bridge configurations & Below \eqref{def:EMP}\\
									$\Ps$& space of stationary probability measures on bridge configurations &Sec. \ref{subsec:entropy}\\
									${\mc U_{\text{meas}}(P)}$& family of neighborhoods of $P$&\eqref{def:N(P)}\\
									$	U_{F, P}$& cylinder set &\eqref{topology-basis}\\

			$\ent(P)$&specific relative entropy of $P$&\eqref{def:entP}\\
			{$H_\La(\CC)$} & energy of $\CC\in \ms{Conf}(\La)$ in terms of electric field & \eqref{Hdef} \\
									$\Neut(\La)$ & space of charge-neutral bridge configurations $\CC$ in $\La$& \eqref{def:neut}\\
			$\W(P)$& expected specific energy &\eqref{wDef}\\
				$\W^M(P)$& expected specific energy truncated at level $M$&\eqref{wDef1}\\
					$\WRM(P)$& truncated expected specific energy in a finite domain &\eqref{def:WRM}\\
			$\tH_{\b, K}(\CC)$ & minimal energy of compatible configurations restricted to $K$ 	&\eqref{eq:wkdef}\\
			$b_x$& $x$-coordinate of a bridge $b$&above \eqref{def:xrangefunctional}\\ 
	$\psi(b)$& $x$-range functional of a bridge $b$&\eqref{def:xrangefunctional}\\
			$\mu_t(\CC,A)$ & net charge of $\CC$ in the region $A$ in the time slice $t$ &\eqref{def: mu_t}\\
			$\CC'$ & regularized configuration &\eqref{eq:regstep}\\
			 $d_0$&  net {negative} charge to the left of $x_0$ &\eqref{eq:d0}\\
			$\DA(\CC)$ & block average of a configuration $\CC$ &\eqref{def:block average}
		\end{tabular}
	\end{center}
\end{table}

\FloatBarrier
\bibliographystyle{amsalpha}
\bibliography{biblio}

\newcommand{\etalchar}[1]{$^{#1}$}
\providecommand{\bysame}{\leavevmode\hbox to3em{\hrulefill}\thinspace}
\providecommand{\MR}{\relax\ifhmode\unskip\space\fi MR }
\providecommand{\MRhref}[2]{%
  \href{http://www.ams.org/mathscinet-getitem?mr=#1}{#2}
}
\providecommand{\href}[2]{#2}
\begin{thebibliography}{DRMN06}

\bibitem[AJJ10]{aizenman2010symmetry}
M.~Aizenman, S.~Jansen, and P.~Jung, \emph{{Symmetry breaking in quasi-1D
  Coulomb systems}}, Ann. Henri Poincar{\'e} \textbf{11} (2010), no.~8, 1--33.

\bibitem[AM80]{aizenman1980structure}
M.~Aizenman and P.~A. Martin, \emph{{Structure of Gibbs states of one
  dimensional Coulomb systems}}, Comm. Math. Phys. \textbf{78} (1980), no.~1,
  99--116.

\bibitem[BDK{\etalchar{+}}05]{bjelakovic-etal2005sanov}
I.~Bjelakovi\'{c}, J.-D. Deuschel, T.~Kr\"{u}ger, R.~Seiler,
  R.~Siegmund-Schultze, and A.~Szko{\l}a, \emph{A quantum version of {S}anov's
  theorem}, Comm. Math. Phys. \textbf{260} (2005), no.~3, 659--671.

\bibitem[Ber18]{berman2018large}
R.~J. Berman, \emph{On large deviations for {G}ibbs measures, mean energy and
  {G}amma-convergence}, Constructive Approximation \textbf{48} (2018), no.~1,
  3--30.

\bibitem[BG99]{bodineau-guionnet1999}
T.~Bodineau and A.~Guionnet, \emph{About the stationary states of vortex
  systems}, Ann. Inst. H. Poincar\'{e} Probab. Statist. \textbf{35} (1999),
  no.~2, 205--237.

\bibitem[BK94]{brydges1994absence}
D.~C. Brydges and G.~Keller, \emph{{Absence of Debye screening in the quantum
  Coulomb system}}, J. Stat. Phys. \textbf{76} (1994), no.~1, 285--297.

\bibitem[BKSSS04]{bjelakovic-etal2004inventiones}
I.~Bjelakovi\'{c}, T.~Kr\"{u}ger, R.~Siegmund-Schultze, and A.~Szko{\l}a,
  \emph{The {S}hannon-{M}c{M}illan theorem for ergodic quantum lattice
  systems}, Invent. Math. \textbf{155} (2004), no.~1, 203--222.

\bibitem[BL75]{brascamp1975some}
H.~J. Brascamp and E.~H. Lieb, \emph{{Some inequalities for Gaussian measures
  and the long-range order of the one-dimensional plasma}}, Functional
  Integration and Its Applications (A.~M. Arthurs, ed.), Clarendon Press, 1975,
  pp.~1--14.

\bibitem[BM99]{brydges1999coulomb}
D.~C. Brydges and P.~A. Martin, \emph{Coulomb systems at low density: A
  review}, J. Stat. Phys. \textbf{96} (1999), no.~5-6, 1163--1330.

\bibitem[CFS83]{choquard1983two}
Ph. Choquard, P.~J. Forrester, and E.~R. Smith, \emph{The two-dimensional
  one-component plasma at $\gamma$= 2: the semiperiodic strip}, J. Stat. Phys.
  \textbf{33} (1983), no.~1, 13--22.

\bibitem[CGZJ20]{chafai2020macroscopic}
D.~Chafa{\"\i}, D.~Garc{\'\i}a-Zelada, and P.~Jung, \emph{Macroscopic and edge
  behavior of a planar jellium}, Journal of Mathematical Physics \textbf{61}
  (2020), no.~3, 033304.

\bibitem[DB08]{deshpande2008one}
V.~V. Deshpande and M.~Bockrath, \emph{The one-dimensional {W}igner crystal in
  carbon nanotubes}, Nature Phys. \textbf{4} (2008), no.~4, 314--318.

\bibitem[DBGY10]{deshpande2010electron}
V.~V. Deshpande, M.~Bockrath, L.~I. Glazman, and A.~Yacoby, \emph{Electron
  liquids and solids in one dimension}, Nature \textbf{464} (2010), no.~7286,
  209.

\bibitem[Der03]{dereudre2003interacting-brownian}
D.~Dereudre, \emph{Interacting {B}rownian particles and {G}ibbs fields on
  pathspaces}, ESAIM Probab. Stat. \textbf{7} (2003), 251--277.

\bibitem[Deu87]{deuschel87diffusions}
J.-D. Deuschel, \emph{Infinite-dimensional diffusion processes as {G}ibbs
  measures on {$C[0,1]^{{\bf Z}^d}$}}, Probab. Theory Related Fields
  \textbf{76} (1987), no.~3, 325--340.

\bibitem[DHLM21]{dereudre}
D.~Dereudre, A.~Hardy, T.~Lebl{\'e}, and M.~Maida, \emph{{DLR} equations and
  rigidity for the {Sine}-beta process}, Comm. Pure Appl. Math. \textbf{74}
  (2021), no.~1, 172--222.

\bibitem[DLR20]{vaios}
P.~Dupuis, V.~Laschos, and K.~Ramanan, \emph{Large deviations for
  configurations generated by {G}ibbs distributions with energy functionals
  consisting of singular interaction and weakly confining potentials},
  Electron. J. Probab. \textbf{25} (2020), Paper No. 46, 41.

\bibitem[DPRZ02]{daipra-roelly-zessin2002}
P.~Dai~Pra, S.~Roelly, and H.~Zessin, \emph{A {G}ibbs variational principle in
  space-time for infinite-dimensional diffusions}, Probab. Theory Related
  Fields \textbf{122} (2002), no.~2, 289--315.

\bibitem[DRMN06]{deroeck-maes-netocny2006}
W.~De~Roeck, C.~Maes, and K.~Neto\v{c}n\'{y}, \emph{Quantum macrostates,
  equivalence of ensembles, and an {$H$}-theorem}, J. Math. Phys. \textbf{47}
  (2006), no.~7, 073303, 12 pp.

\bibitem[Dur19]{durrett2019probability}
R.~Durrett, \emph{Probability: theory and examples}, Cambridge University
  Press, 2019.

\bibitem[DZ98]{dz98}
A.~Dembo and O.~Zeitouni, \emph{Large deviations techniques and applications},
  second ed., Springer, New York, 1998.

\bibitem[EG99]{eichelsbacher-grunwald99}
P.~Eichelsbacher and M.~Grunwald, \emph{Exponential tightness can fail in the
  strong topology}, Statist. Probab. Lett. \textbf{41} (1999), no.~1, 83--86.

\bibitem[FJS83]{forrester1983two}
P.~J. Forrester, B.~Jancovici, and E.~R. Smith, \emph{The two-dimensional
  one-component plasma at $\gamma$= 2: Behavior of correlation functions in
  strip geometry}, J. Stat. Phys. \textbf{31} (1983), no.~1, 129--140.

\bibitem[For91]{forrester1991finite}
P.~J. Forrester, \emph{{Finite-size corrections to the free energy of Coulomb
  systems with a periodic boundary condition}}, J. Stat. Phys. \textbf{63}
  (1991), no.~3, 491--504.

\bibitem[Fri87]{fritz87gradient}
J.~Fritz, \emph{Gradient dynamics of infinite point systems}, Ann. Probab.
  \textbf{15} (1987), no.~2, 478--514.

\bibitem[Geo93]{georgii1}
H.-O. Georgii, \emph{{Large deviations and maximum entropy principle for
  interacting random fields on ${\bf Z}^d$}}, Ann. Probab. (1993), 1845--1875.

\bibitem[Geo94]{georgii1994equivalence}
\bysame, \emph{Large deviations and the equivalence of ensembles for {G}ibbsian
  particle systems with superstable interaction}, Probab. Theory Related Fields
  \textbf{99} (1994), no.~2, 171--195.

\bibitem[Gin65]{ginibre1965reduced}
J.~Ginibre, \emph{{Reduced density matrices of quantum gases. I. Limit of
  infinite volume}}, J. Math. Phys. \textbf{6} (1965), 238.

\bibitem[Gin71]{ginibre}
\bysame, \emph{{Some applications of functional integration in statistical
  mechanics}}, Statistical Mechanics and Quantum Field Theory (C.~de~Witt and
  R.~Stora, eds.), Gordon and Breach, New York, 1971, pp.~327--429.

\bibitem[GLM02]{gallavotti-lebowitz-mastropietro2002}
G.~Gallavotti, J.~L. Lebowitz, and V.~Mastropietro, \emph{Large deviations in
  rarefied quantum gases}, J. Statist. Phys. \textbf{108} (2002), no.~5-6,
  831--861, Dedicated to David Ruelle and Yasha Sinai on the occasion of their
  65th birthdays.

\bibitem[GZ93]{georgii2}
H.-O. Georgii and H.~Zessin, \emph{Large deviations and the maximum entropy
  principle for marked point random fields}, Probab. Theory Related Fields
  \textbf{96} (1993), no.~2, 177--204.

\bibitem[GZ19]{garcia2017large}
D.~Garc\'{\i}a-Zelada, \emph{A large deviation principle for empirical measures
  on {P}olish spaces: application to singular {G}ibbs measures on manifolds},
  Ann. Inst. Henri Poincar\'{e} Probab. Stat. \textbf{55} (2019), no.~3,
  1377--1401.

\bibitem[JJ14]{jansen2014wigner}
S.~Jansen and P.~Jung, \emph{Wigner crystallization in the quantum 1{D} jellium
  at all densities}, Comm. Math. Phys. \textbf{331} (2014), no.~3, 1133--1154.

\bibitem[JLS08]{jansen2008laughlin}
S.~Jansen, E.~H. Lieb, and R.~Seiler, \emph{Laughlin's function on a cylinder:
  plasma analogy and representation as a quantum polymer}, Physica Status
  Solidi (B) \textbf{245} (2008), no.~2, 439--446.

\bibitem[KCZ{\etalchar{+}}16]{PhysRevB.94.115417}
I.~Kyl\"anp\"a\"a, F.~Cavaliere, N.~Traverso Ziani, M.~Sassetti, and
  E.~R\"as\"anen, \emph{{Thermal effects on the Wigner localization and Friedel
  oscillations in many-electron nanowires}}, Phys. Rev. B \textbf{94} (2016),
  115417.

\bibitem[KRS20]{schlein}
K.~Kirkpatrick, S.~Rademacher, and B.~Schlein, \emph{A large deviation
  principle in many-body quantum dynamics}, Arxiv preprint arXiv:2010.13754
  (2020).

\bibitem[KS08]{neumann}
A.~S. Kim and Z.~Shen, \emph{The {N}eumann problem in {$L^p$} on {L}ipschitz
  and convex domains}, J. Funct. Anal. \textbf{255} (2008), no.~7, 1817--1830.

\bibitem[Lan77]{lang1977}
Reinhard Lang, \emph{Unendlich-dimensionale {W}ienerprozesse mit
  {W}echselwirkung. {I}. {E}xistenz}, Z. Wahrscheinlichkeitstheorie und Verw.
  Gebiete \textbf{38} (1977), no.~1, 55--72.

\bibitem[LI73]{lanford1973}
O.~E. Lanford~III, \emph{Entropy and equilibrium states in classical
  statistical mechanics}, Statistical mechanics and mathematical problems,
  Springer, 1973, pp.~1--113.

\bibitem[LLS00]{lebowitz-lenci-spohn2000}
J.~L. Lebowitz, M.~Lenci, and H.~Spohn, \emph{Large deviations for ideal
  quantum systems}, J. Math. Phys. \textbf{41} (2000), no.~3, 1224--1243,
  Probabilistic techniques in equilibrium and nonequilibrium statistical
  physics.

\bibitem[LP76]{lebowitz-presutti76}
J.~L. Lebowitz and E.~Presutti, \emph{Statistical mechanics of systems of
  unbounded spins}, Comm. Math. Phys. \textbf{50} (1976), no.~3, 195--218.

\bibitem[LP17]{last2017lectures}
G.~Last and M.~Penrose, \emph{{Lectures on the Poisson Process}}, Cambridge
  University Press, Cambridge, 2017.

\bibitem[LS17]{serfInv}
T.~Lebl{\'e} and S.~Serfaty, \emph{Large deviation principle for empirical
  fields of {L}og and {R}iesz gases}, Invent. Math. \textbf{209} (2017),
  1--113.

\bibitem[LW20]{liu2020large}
W.~Liu and L.~Wu, \emph{Large deviations for empirical measures of mean-field
  gibbs measures}, Stochastic Process. Appl. \textbf{130} (2020), no.~2,
  503--520.

\bibitem[MM08]{meyer2008wigner}
J.~S. Meyer and K.~A. Matveev, \emph{Wigner crystal physics in quantum wires},
  J. Phys. Condens. Matter \textbf{21} (2008), no.~2, 023203.

\bibitem[MP10]{bmBook}
P.~M\"orters and Y.~Peres, \emph{Brownian motion}, Cambridge University Press,
  Cambridge, 2010.

\bibitem[NR04]{netocny-redig2004}
K.~Neto\v{c}n\'{y} and F.~Redig, \emph{Large deviations for quantum spin
  systems}, J. Statist. Phys. \textbf{117} (2004), no.~3-4, 521--547.

\bibitem[Ons39]{onsager}
L.~Onsager, \emph{Electrostatic interaction of molecules}, J. Phys. Chem.
  \textbf{43} (1939), no.~2, 189--196.

\bibitem[ORB11]{ogata-reybellet2011}
Y.~Ogata and L.~Rey-Bellet, \emph{Ruelle-{L}anford functions and large
  deviations for asymptotically decoupled quantum systems}, Rev. Math. Phys.
  \textbf{23} (2011), no.~2, 211--232.

\bibitem[Osa12]{osada2012}
H.~Osada, \emph{Infinite-dimensional stochastic differential equations related
  to random matrices}, Probab. Theory Related Fields \textbf{153} (2012),
  no.~3-4, 471--509.

\bibitem[PS17]{petrache}
M.~Petrache and S.~Serfaty, \emph{Next order asymptotics and renormalized
  energy for {R}iesz interactions}, J. Inst. Math. Jussieu \textbf{16} (2017),
  no.~3, 501--569.

\bibitem[RAS15]{rassoul2015course}
F.~Rassoul-Agha and T.~Sepp{\"a}l{\"a}inen, \emph{A course on large deviations
  with an introduction to {G}ibbs measures}, vol. 162, American Mathematical
  Society, Providence, RI, 2015.

\bibitem[RCL18]{ray2018importance}
U.~Ray, G.~K.-L. Chan, and D.~T. Limmer, \emph{Importance sampling large
  deviations in nonequilibrium steady states. i}, J. Chem. Phys. \textbf{148}
  (2018), no.~12, 124120.

\bibitem[RS16]{rouSe}
N.~Rougerie and S.~Serfaty, \emph{Higher-dimensional {C}oulomb gases and
  renormalized energy functionals}, Comm. Pure Appl. Math. \textbf{69} (2016),
  no.~3, 519--605.

\bibitem[Rue70]{ruelle70}
D.~Ruelle, \emph{Superstable interactions in classical statistical mechanics},
  Comm. Math. Phys. \textbf{18} (1970), 127--159.

\bibitem[Rue76]{ruelle76}
\bysame, \emph{Probability estimates for continuous spin systems}, Comm. Math.
  Phys. \textbf{50} (1976), no.~3, 189--194.

\bibitem[Ser15]{serfBook}
S.~Serfaty, \emph{Coulomb gases and {G}inzburg-{L}andau vortices}, European
  Mathematical Society, Z\"urich, 2015.

\bibitem[Spo86]{spohn1986}
H.~Spohn, \emph{Equilibrium fluctuations for interacting {B}rownian particles},
  Comm. Math. Phys. \textbf{103} (1986), no.~1, 1--33.

\bibitem[Spo87]{spohn1986dyson}
\bysame, \emph{Interacting {B}rownian particles: a study of {D}yson's model},
  Hydrodynamic behavior and interacting particle systems ({M}inneapolis,
  {M}inn., 1986), IMA Vol. Math. Appl., vol.~9, Springer, New York, 1987,
  pp.~151--179.

\bibitem[SS12]{ginzburg}
E.~Sandier and S.~Serfaty, \emph{From the {G}inzburg-{L}andau model to vortex
  lattice problems}, Comm. Math. Phys. \textbf{313} (2012), no.~3, 635--743.

\bibitem[SS15]{sandSerf}
\bysame, \emph{1{D} log gases and the renormalized energy: crystallization at
  vanishing temperature}, Probab. Theory Related Fields \textbf{162} (2015),
  no.~3-4, 795--846.

\bibitem[{\v{S}}WK04]{vsamaj2004translation}
L.~{\v{S}}amaj, J.~Wagner, and P.~Kalinay, \emph{Translation symmetry breaking
  in the one-component plasma on the cylinder}, J. Stat. Phys. \textbf{117}
  (2004), no.~1, 159--178.

\bibitem[Tsa16]{tsai2016}
L.-C. Tsai, \emph{Infinite dimensional stochastic differential equations for
  {D}yson's model}, Probab. Theory Related Fields \textbf{166} (2016), no.~3-4,
  801--850.

\bibitem[Wig34]{wigner1934interaction}
E.~Wigner, \emph{On the interaction of electrons in metals}, Phys. Rev.
  \textbf{46} (1934), no.~11, 1002.

\end{thebibliography}

\end{document}